\documentclass[10pt]{article}  
\usepackage{amssymb}
\usepackage{amsmath,amsbsy,amsfonts,amsthm}
\usepackage{graphicx}
\usepackage{geometry}
\usepackage{wrapfig}
\usepackage{float}
\usepackage{subfig}
\usepackage{setspace}

\usepackage[english]{babel}

\newtheorem{theorem}{Theorem}[section]
\newtheorem{lemma}[theorem]{Lemma}
\newtheorem{proposition}[theorem]{Proposition}
\newtheorem{corollary}[theorem]{Corollary}
\newtheorem{remark}[theorem]{Remark}

   \newcommand{\R}{\mathbb{R}}
   \newcommand{\spn}[1]{\langle#1\rangle}
   \newcommand{\N}{\mathbb{N}}

\begin{document}

\numberwithin{equation}{section}

\title{Self-Similar Solutions to the Curvature Flow and its Inverse on the 2-dimensional Light Cone}

\author{F\'abio Nunes da Silva\footnote{ Universidade de Bras\'{\i}lia,
 Department of Mathematics,
 70910-900, Bras\'{\i}lia-DF,  and Universidade Federal do Oeste da Bahia, Brazil,  fabionuness@ufob.edu.br. 
Supported by the Universidade Federal do Oeste da Bahia during his graduate program at the Universidade de Bras\'ilia.} \qquad Keti Tenenblat\footnote{ Universidade de Bras\'{\i}lia,
 Department of Mathematics,
 70910-900, Bras\'{\i}lia-DF, Brazil, K.Tenenblat@mat.unb.br Partially supported by CNPq Proc. 312462/2014-0, Ministry of Science and Technology, Brazil and CAPES/Brazil-Finance Code 001.}
  }

\date{}

\maketitle{}

\begin{abstract}
We show that the solutions to the curvature flow (CF) for curves on the 2-dimensional light cone are in correspondence with the solutions to the inverse curvature flow (ICF). We prove that 
 the ellipses and the hyperboles are the only curves that evolve under homotheties. The ellipses are the only closed ones and they are ancient solutions. We show that a spacelike curve on the cone is a self-similar solution to the CF (resp. (ICF)) if, only if, its curvature (resp. inverse of its curvature) differs by a constant $c$ from being the inner product between its tangent vector field and a fixed vector $v$ of the 3-dimensional Minkowski space. The curve is a soliton solution when $c=0$. 
We prove that, for each vector $v$ there exist a 2-parameter family of self-similar solutions to the CF and to the ICF, on the light cone.  Moreover,  at each  end of such a curve the curvature is either  unbounded or it tends  to $0$ or  to the constant $c$. Explicitly given soliton solutions are included and some self-similar solutions on the light cone, are visualized.  
  \newline
\noindent \emph{Keywords}: Curvature Flow, Inverse  Curvature FLow,  Light Cone, self-similar solutions, soliton solutions, ancient solutions. 
\newline
\noindent \emph{Mathematics Subject Classification 2010}: 53C44, 53C50, 37E35.
\end{abstract}

\section{Introduction}

We consider the 3-dimensional Minkowski space as $\mathbb{R}_1^{3}=(\mathbb{R}^{3}, \langle, \rangle )$, where $\mathbb{R}^3$ is the 3-dimensional vector space and $\langle, \rangle$ is the Minkowski metric defined by $\langle u,v \rangle=-u_1v_1+u_2v_2+u_3v_3.$ We define the \it light cone \rm as the lightlike surface $Q^{2}:= \{p=(p_1,p_2,p_3) \in \R_{1}^{3}\setminus \{0\}:\spn{p,p}=0\}.$

Let $\displaystyle X:I\subset \R \rightarrow Q^{2}\subset \mathbb{R}_1^{3}$ be a spacelike curve parametrized by arc length $s$. The curve $X$ is characterized by a trihedrom $\{X(s),T(s),Y(s)\}$, $s \in I$, where $X^{\prime}(s)=T(s)$ is the unit tangent vector field and $Y(s)$ is the unique lightlike vector field orthogonal to $T(s)$,  such that $\spn{X(s),Y(s)}=1$. The {\it curvature} $k(s)$ of $X(s)$, at $s \in I$, is defined by
\begin{equation}\label{defk}
 T^{\prime}(s)=k(s)X(s)-Y(s).
 \end{equation}
 The curvature measures how much the curve $X$ bends from  a parable, whose curvature vanishes. When the curvature is a positive (resp. negative) constant function, then $X(s)$ is a hyperbole (resp. ellipse). Morover we say that $Y(s)$ is the \it associated curve \rm to $X(s)$. If $k(s)\neq 0$, $s \in I$, then $Y(s)$ is a spacelike curve and its  curvature is given by $\tilde{k}(s)=1/k(s)$.

The light cone has two connected components, namely $Q^{2}=Q^{2}_{+}\cup Q^{2}_{-}$, where $Q_{+}^{2}=\{p \in Q^{2}:p_1>0\}$ and $Q_{-}^{2}=\{p \in Q^{2}:p_1<0\}$.  Let $X(s)$ be a spacelike curve on $Q_{+}^{2}$, then it follows from $\spn{X(s),Y(s)}=1$ that $Y(s)$ is a curve on $Q_{-}^{2}$ , hence  $-Y(s) \in Q_{+}^{2}$ for each $s$. Without loss of generality, we will consider  only  $Q_{+}^{2}$, since the same results hold for $Q_{-}^{2}$.

A 1-parameter family of spacelike curves $\displaystyle \hat{X}: I\times J \rightarrow Q_{+}^{2}\subset \mathbb{R}_1^{3}$ is a \it solution to the curvature flow  \rm (CF) (resp. \it inverse curvature flow \rm (ICF)), with initial condition $X(u)$, $u\in I$, if  
\begin{equation}\label{cficf}
\left\{\begin{array}{ll}
\displaystyle{<\frac{\partial }{\partial t}\hat{X}^{t}(\cdot), \hat{Y}^t(\cdot)>}=\hat{k}^{t}(\cdot)\\\
\hat{X}^{0}(\cdot)=X(\cdot),
\end{array}\right. 
\qquad 
\mbox( resp. ) \left\{\begin{array}{ll}
\displaystyle{<\frac{\partial }{\partial t}\hat{X}^{t}(\cdot), \hat{Y}^t(\cdot)>}=-\frac{1}{\hat{k}^{t}(\cdot)}\\\
\hat{X}^{0}(\cdot)=X(\cdot),
\end{array}\right. 
\end{equation}
where $\hat{k}^{t}(\cdot)=\hat{k}(\cdot,t)$ is the curvature of $\hat{X}^{t}(\cdot)=\hat{X}(\cdot,t)$ and $\hat{Y}^{t}(\cdot)=\hat{Y}(\cdot,t)$ is the lightlike vector field associated to $\hat{X}^{t}(\cdot)$ for each $t\in J$. 
When $X(u)$ is a parable i.e. $k\equiv 0$, then the family $\hat{X}^{t}(u)=X(u)$, for all $t$,  is a {\it trivial solution} to the CF. 

The definition above was motivated by the curve shortening  flow  for curves on a 2-dimensional manifold $M^2$, where one considers the inner product
$<\partial/\partial t\;\hat{X}^{t}, \hat{N}^t>=\hat{k}^{t}$ 
where $N^t$ is the unit vector field normal to the curve.  This is a gradient type flow for the length functional. The curve shortening flow for curves on the 2-dimensional Euclidean space was studied by several authors in  \cite{Abresch}, \cite{Epstein1} 
-  \cite{Grayson1} and \cite{Halldorsson}. In particular,  the solutions that evolve by isometries and/or homotheties were investigated. These are  the so called { self-similar solutions} of the  curve shortening flow and they are called { solitons}, when the curve evolves  by isometries. The importance of such flows is due to the fact, that after partial results obtained by several authors,  Angenent \cite{Angenent} proved that for closed curves, under general  conditions, the curve shortening flow turns into a self-similar  solution and eventually collapses into a point. Halldorsson  \cite{Halldorsson} gave a complete description of the self-similar solutions for curves on the plane. 
An increased interest also appeared in investigating the flows for curves on the plane, when one replaces the curvature $k$  by a function of the curvature such as $ 1/k^\alpha$ in \cite{Urbas} and    \cite{Urbas1}, or $k^\alpha/\alpha $ studied in \cite{Andrews}. Considering 
the function $-1/k$,  Drugan et al. \cite{Drugan} investigated the curves on the plane that evolve by translations and Andrews \cite{Andrews} showed that the only simple closed curves on the plane that evolve by homotheties are circles.  There are a  few results in \cite{Gage2}, \cite{Grayson} and  \cite{Ma} 
 for curves on an ambient space wich  is not the  Euclidean plane.  Moreover, Dos Reis and Tenenblat \cite{DosReis} characterized and described all the soliton solutions of the curve shortening flow on the sphere and  Nunes da Silva and Tenenblat \cite{Silva} described all the soliton solutions of this flow, on the 2-dimensional hyperbolic space.   
    
Halldorsson \cite{Halldorsson1} in 2015, considered curves on the Minkowski plane and classified all the self similar solutions of the curvature flow. 
He also proved that this flow is not necessarily length decreasing. 
In this paper, we study curvature flows for curves on the light cone $Q^2_+\subset \R^3_1$. 

Considering curves whose curvature does not vanish, we show that studying the  solutions to the curvature flow (CF) on the light cone is equivalent to studying the solutions to the inverse curvature flow (ICF).
We prove, in Theorems \ref{c3t2} (resp. \ref{c3t3}), that a spacelike curve on the light cone is a self-similar 
  solution to the CF (resp. (ICF)) if, only if, its curvature function $k$ (resp. $1/k$, for $k\neq 0$) differs by a constant $c$ from being the inner product between its tangent vector field and a fixed vector $v\in \R^3_1\setminus\{0\}$. 
 
 We investigate the {\it self-similar solutions} to the flows i.e. the curves that evolve by isometries  and/or homotheties of $Q^2_+$. When the constant $c=0$, then the curve evolves only by isometries and it is called a {\it soliton solution}. We will show that  curves whose curvature is a non zero constant (ellipses and hyperboles) are the only curves that evolve by homotheties. The ellipses are the only closed ones and they are ancient solutions (see \cite{Bourni}), i.e. they evolve for time $t\in (-\infty, A)$.  
 
 We prove that the self-similar flows are characterized in terms of a system of ordinary differential equations, which admits initial condition on three disjoint sets.  We prove  existence results, which show that for each vector $v\in \R^3_1\setminus\{0\}$, there exists a 2-parameter family of self-similar solutions to the CF (and consequently to the ICF) on $Q^2_+$. There are three classes of such solutions which are associated to the type of the vector $v$. 
  Moreover, considering non trivial solutions to the CF we prove that the curvature may vanish at most on two points and therefore the corresponding solutions to the ICF may have at most three connected components. We study the behaviour of the curvature function at each end of the maximal interval of definition of the curve. 
  
  In what follows we state our main results. We start establishing the correspondence between the CF and the ICF on $Q_{+}^{2}$. 

\begin{remark}\label{c3t0}
	\rm Let $X: I \rightarrow Q_{+}^{2}$, $u \in I$, be a spacelike curve and let $Y(u)$ be the curve associated to $X$. Let $\hat{X}^{t}(u)=\hat{X}(u,t)$, $\hat{X}:I\times J \rightarrow Q_{+}^2$, $(u,t) \in I\times J$, $0 \in J$ be a 1-parameter family of curves with non vanishing curvature for all $t$  
and let $\hat{Y}^{t}(\cdot)=\hat{Y}(\cdot,t)$ be the associated curve to $\hat{X}^{t}(\cdot)$. For each $t$, $\hat{Y}^t$ is a curve on $Q^2_-$. Moreover, $<\hat{X}(u,t),\hat{Y}(u,t)>=1$. Taking the derivative with respect to $t$ and considering the curves on $Q^2_+$, we conclude that   $\hat{X}^{t}(\cdot)$ is a 
solution to the CF, with initial condition $X(u)$ 
if, and only if, $-\hat{Y}(\cdot,-t)$ is a solution to the 
ICF, with initial condition $-Y(u)$.
\end{remark}

We investigate solutions to the CF and the ICF that
  evolve by  homotheties and/or isometries of $Q_{+}^{2}$. 
We remark that an isometry of $Q_{+}^{2}$ is an element of the Lie group $O_1(3)$, acting on $\R^3_1$,  that preserves $Q_{+}^{2}$.  
	Let	$\displaystyle \hat{X}: I\times J \rightarrow Q_{+}^{2}$ be a solution to the CF (resp. ICF) on $Q_{+}^{2}$, with initial condition $\displaystyle X: I \rightarrow Q_{+}^{2}$.  
	The curve $X$ is  a  self-similar solution  to the CF (resp. ICF) if $\hat{X}^{t}(s)=f(t)M(t)X(s)$,  where $f(t)>0$ is a smooth map with $f(0)=1$  and $M(t)$, $t \in J$ is a family of   isometries of $Q_{+}^{2}$, such that $M(0)=Id$ is the identity map.  
	 If $f(t)\equiv 1$ for all $t \in J$, then $X$ is a  soliton solution to the CF (resp. ICF). Our next two theorems classify the solutions to the CF and ICF that evolve by homotheties on  $Q^{2}_{+}$. 

\begin{theorem}\label{c3t1}
	Let $X: I \rightarrow Q^{2}_{+}$ be a spacelike curve parametrized by arc length $s$ with curvature $k(s)\not\equiv 0$ and let $\hat{X}(s,t)=f(t)X(s), \, (s,t) \in I \times J$, $f(t)>0,\; f(0)=1$ be an evolution of $X$ by homotheties. The family $\hat{X}(s,t)$ is a solution to the CF if, and only if, the curvature $k$ of $X$  is constant, $f(t)=\sqrt{2kt+1}$ 
	and $\hat{k}(x,t)= k/(2kt+1)$. 
In particular
\begin{enumerate}
\item If $k<0$ then  $X$ is an ellipse of  $Q_+^2$, with curvature $k$ and it is an ancient solution, with 
$J=(-\infty,-\frac{1}{2k})$. At $t= -\frac{1}{2k}$, $\hat{X}^t$ collapses into the origin of $R^3_1$.   
\item If $k>0$ then $X$ is a hyperbole of $Q_{+}^{2}$, with curvature $k$, and $J=(-\frac{1}{2k},+\infty)$.  
\end{enumerate}
 \end{theorem}

 One can see, from Theorem \ref{c3t1}, that the curvature flow on $Q_{+}^{2}$ is not always a curve shortening flow, as it occurs on the Minkowski plane (see \cite{Halldorsson1}). 
In fact, if $\hat{X}(s,t)=f(t)X(s)$, $f(t)>0,\, f(0)=1$, is a solution to the CF on $Q_{+}^2$, then the arc length of $X^{t}(s)$ is given by $h(t)=f(t)(s_1-s_0)$, for $s_0<s_1$, and $h^{\prime}(t)={k}/({2kt+1})\,(s_1-s_0)$. Therefore, when  $X(s)$ is an ellipse (resp. hyperbole) the arc length of $X^{t}(s)$ decreases (resp. increases) along  the flow.  

\begin{theorem}\label{c3t1I}
Let $X: I \rightarrow Q^{2}_{+}$ be a spacelike curve parametrized by arc length $s$ whose curvature  $k(s)$ does not vanish and let $\hat{X}(s,t)=f(t)X(s), \, (s,t) \in I \times J$, $f(t)>0,\; f(0)=1$ be an evolution of $X$ by homotheties.
The family $\hat{X}(s,t)$ satisfy the ICF if, and only if, 
the curvature $k$ of $X$ is constant, $ f(t)={1}/{\sqrt{{2}t/k+1}}$ 
and  $\hat{k}(s,t)=2t+k$.  In particular 
\begin{enumerate}
\item If $k<0$, then $X$ is an ellipse of $Q_+^2$, with curvature $k$. 
It is an ancient solution with  $J=\left(-\infty,-\frac{k}{2} \right)$, evolving from the origin of $\R^3_1$. 
\item  If $k>0$, then $X$ is a hyperbole of  $Q_+^2$, with curvature $k$,
and $ J=\left(-\frac{k}{2},+\infty \right)$.
\end{enumerate}
\end{theorem}

\begin{corollary}\label{closed}
The ellipses are the only closed curves on the light cone that evolve by homotheties, along the curvature flow or the inverse curvature flow. They are ancient solutions  that collapse into $0\in \R^3_1$ along the CF and they evolve from $0\in \R^3_1$ along the ICF.
\end{corollary}

The following two results provide a characterization of the self-similar solutions to the CF and to the ICF.

\begin{theorem} \label{c3t2}
	Let $\displaystyle X: I \rightarrow Q^{2}_{+}$ be a spacelike curve parametrized by arc length $s\in I$. Then $X$ is a self-similar solution to the curvature flow on $Q_{+}^{2}$ if, and only if, there exist a vector $v \in \mathbb{R}^{3}_1\setminus \{0\}$ and a constant $c \in \R$ such that
	\begin{equation}\label{eq2}
	c+\langle T(s),v\rangle=k(s),
	\end{equation}
	where $T$ is the unit tangent vector field and $k$ is the curvature  of $X$. In particular, $X$ is a soliton solution to the CF whenever $c=0$.
\end{theorem}

\begin{theorem} \label{c3t3}
	Let $\displaystyle X: I \rightarrow Q^{2}_{+}$ be a spacelike curve parametrized by arc length $s$, such that $k(s)\neq0$ for all $s \in I$. Then $X$ is a self-similar solution to the inverse  
	curvature flow on $Q_{+}^{2}$ if, and only if, there exist a 
	vector $v \in \mathbb{R}^{3}_1\setminus \{0\}$ and $c \in \R$ such that
	\begin{equation}\label{eq2n}
	c+\langle T(s),v\rangle=\frac{1}{k(s)},
	\end{equation}
	where $T$ is the unit tangent vector field and $k$ is the curvature  of $X$. In particular, $X$ is a soliton solution to the ICF whenever $c=0$. 
\end{theorem}

Observe that a parable, an ellipse or a hyperbole is a  self-similar solution to the CF satisfying also \eqref{eq2}. In fact, considering  the tangent vector field $T$, since it is a planar curve, there exits a vector $v\in{\R^3_1}\setminus \{0\}$ such that $\displaystyle \langle T,v\rangle=0$, hence \eqref{eq2} is satisfied for the constant $c=k$. Similarly, by considering $c=1/k$, an ellipse or a hyperbole satisfy \eqref{eq2n} hence it is a self-similar solution to the ICF. These curves will be called {\it trivial solutions} to the CF and to the ICF.

As a consequence of the characterizations given by Theorems \ref{c3t2} and \ref{c3t3}, we can show that obtaining self-similar solutions to the CF and 
to the ICF correspond to obtaining solutions to  systems of ODEs (see \eqref{s2},\; 
 \eqref{s3} and   Propositions \ref{c4p1},\;\ref{c4p2} and \ref{propinvers}). 
A long sequence of lemmas, providing properties of the solutions of these systems, will prove the  following existence theorem. 
\begin{theorem}\label{c3t4}
	For any $v \in \mathbb{R}^{3}_1\setminus\{0\}$ and $c \in \R$, there is a 2-parameter family of curves $X$, which are non-trivial self-similar solutions (soliton solutions when $c=0$) to the CF,  on the 2-dimensional light cone $Q^2_+$.  There are three classes of such solutions corresponding to each type of the vector $v$.  Moreover, the curvature function of $X$ has at most two zeros and 
 at each end, the curvature function is either unbounded or it tends to one of the following constants $\{c,0\}$. 
Each curve $-Y$ associated to $X$ is a self-similar solution to the ICF and it has at most three connected components. 
\end{theorem}

The proofs of the theorems stated above are given in Section 2, where we also provide explicit soliton solutions.  
In Section 3, some self-similar solutions to the CF and to the ICF, on the light cone, are visualized.

\section{Proof of the main results}
 We start by  providing some properties 
of a self-similar evolution of a curve on $Q^2_+$. In order to do so, we need  
 the expression of the curvature function $k$ of a curve  $X$ on the light cone,  parametrized by an arbitrary parameter. The curvature is given by 
 (see \cite{Liu}) 
\begin{equation}\label{c1e5}
k= \frac{( \spn{X',X''})^{2}-\spn{X',X'}\spn{X'',X''}}
{2(\spn{X'X'})^3}.
\end{equation}
 
\begin{proposition} \label{c3p1}
	Let $X: I \rightarrow Q^{2}_{+}$ be a spacelike curve parametrized by arc length $s$, with curvature $k(s)$ and associated curve $Y(s)$. Let $\hat{X}(s,t)=f(t)M(t)X(s), \, (s,t) \in I \times J$ be a self-similar evolution of $X$. Then, for each $t \in J$,
	\begin{equation}\label{TYkhat}
	\hat{T}^{t}(s)=M(t)T(s), \hspace{0.5 cm} \hat{Y}^{t}(s)=\frac{1}{f(t)}M(t)Y(s)\hspace{0.5 cm},  \hspace{0.5cm}\hat{k}^{t}(s)=\frac{k(s)}{f^2(t)},
	\end{equation}
	 where $\hat{T}^{t}(s)$ is the unit vector field tangent to $\hat{X}^{t}(s)$, $\hat{Y}^{t}(s)$ is the associated curve to $\hat{X}^{t}(s)$ and $\hat{k}^{t}(s)$ is the curvature of $\hat{X}^{t}(s)$.
\end{proposition}

\begin{proof}
	It follows from the expression of $\hat{X}$  that  
	$\hat{T}^{t}(s)=M(t)T(s)$. Moreover, considering $\hat{Y}(s)=M(t)Y(s)/f(t)$, then  
		$\spn{\hat{Y}^{t}(s),\hat{Y}^{t}(s)}=0$, 
		$\spn{\hat{Y}^{t}(s),\hat{T}^{t}(s)}=0$ and  
		$\spn{\hat{Y}^{t}(s),\hat{X}^{t}(s)}=1$. 	
 Hence $ \hat{Y}^{t}(s)$ is  the curve associated to $\displaystyle \hat{X}^{t}(s)$. We conclude from  (\ref{c1e5}) that  
	\[ 
		\hat{k}^{t}(s)
		=-\frac{\spn{M(t)T^{\prime}(s),M(t)T^{\prime}(s)}}{2f^{2}(t)}=	
\frac{k(s)}{f^2(t)},
\]		
where in the last equality we used  $\left\spn{M(t)T^{\prime}(s),M(t)T^{\prime}(s) \right}=-2k(s)$, which follows  from \eqref{defk}.  
\end{proof}	
  
Let $X: I \rightarrow Q^{2}_{+}$ be a spacelike curve parametrized by arc length $s$ such that $k(s)\neq 0$ for all $s$. It follows from Remark  \ref{c3t0} and Proposition \ref{c3p1} that $\hat{X}(s,t)=f(t)M(t)X(s)$ satisfy the CF in $Q_{+}^{2}$ if, and only if, $\displaystyle -\hat{Y}(s,-t)=-\frac{1}{f(-t)}M(-t)Y(s)$ satisfy the ICF in $Q_{+}^{2}$ i.e. $X$ is a self-similar solution to the CF if, and only if, $-Y(s)$ is a self-similar solution to the ICF. In this context, it is sufficient to study the behaviour  of the self-similar solution to the CF.

\begin{proof}[\bf Proof of Theorem \ref{c3t1}\rm]
 Consider $\hat{X}^{t}(s)=f(t)X(s)$, $(s,t) \in I\times J$, where  $f(t)>0$ and $f(0)=1$. Then Proposition \ref{c3p1} implies that $ \hat{Y}^{t}(s)=Y(s)/f(t)$ and $ \hat{k}^{t}(s)=k(s)/f^2(t)$. 
If $\hat{X}$ is a solution to the CF then  
\begin{equation*}
\frac{k(s)}{f^2(t)}=\left\spn{\frac{\partial}{\partial t}\hat{X}^{t}(s),\hat{Y}^{t}(s)\right}=\frac{f^{\prime}(t)}{f(t)}.
\end{equation*}
Therefore, $ f^{\prime}(t)f(t)=k(s)$ for all $(s,t)$ and hence $k$ is constant and $f(t)=\sqrt{2kt+1}$, since  $f(0)=1$. Hence,  the curve is an ellipse (resp. hyperbole) if $k<0$ (resp. $k>0$) and the domain $J$ is determined by the function $f$ and the sign of $k$. Moreover, since $f(-k/2)=0$, it follows that $X^t$ collapses into the origin of $\mathbb{R}^3_1$. 

Conversely, if $k$ is constant, by considering $f(t)=\sqrt{2kt+1}$ 
and $\hat{X}^{t}(s)=f(t)X(s)$, a straightforward computation using \eqref{TYkhat} shows that $\hat{X}$ is a solution to the CF.    

\end{proof}

The proof of Theorem \ref{c3t1I} is similar to the proof of Theorem \ref{c3t1} and it will be omitted.

\begin{proof}[\bf Proof of Theorem \ref{c3t2}\rm]
	Let $\hat{X}(s,t)=f(t)M(t)X(s)$ be a self-similar evolution of $X$ that satisfy the CF. Taking the derivative of $\hat{X}$ with respect to $t$, it  
	follows from \eqref{TYkhat} that 
	\begin{equation*}
		\frac{k(s)}{{f^{2}(t)}} =\hat{k}(s,t)=\left\spn{\frac{\partial}{\partial t}\hat{X}(s,t),\hat{Y}(s,t)\right}
		=\frac{f^{\prime}(t)}{f(t)}+\left\spn{M^{\prime}(t)X(s),M(t)Y(s)\right}. 
	\end{equation*}
	Therefore, $\displaystyle	k(s)=f(t)f^{\prime}(t)+{f^{2}(t)}\left\spn{M^{\prime}(t)X(s),M(t)Y(s)\right}$ 	for each $(s,t) \in I\times J$  and  hence at $t=0$
	\begin{equation*}
	k(s)=f^{\prime}(0)+\spn{M^{\prime}(0)X(s),Y(s)}.
	\end{equation*}
	$M^{\prime}(0)$ is an element of the Lie algebra $\mathfrak{o}_1(3)$ of the Lie group $O_1(3)$. We consider a basis  of $\mathfrak{o}_1(3)$ 
	\begin{eqnarray*}
		A_1=\left(\begin{array}{ccc}
			0 & 0 & 0  \\
			0 & 0 &  1\\
			0 & -1 &  0
		\end{array}\right),\,\,\,A_2= \left(\begin{array}{ccc}
			0 & 0 & 1  \\
			0 & 0 &  1\\
			1 & -1 &  0
		\end{array}\right)\,\,\,\, \text{and} \,\,\,\,\, 
		A_3=\left(\begin{array}{ccc}
			0 & 1 & 0  \\
			1 & 0 &  0\\
			0 & 0 &  0
		\end{array}\right). 
	\end{eqnarray*}
Then  $M^{\prime}(0)=a_1A_1+a_2A_2+a_3A_3$. Denoting  $X(s)=(x_1(s),x_2(s),x_3(s))$, we have 
	\begin{eqnarray*}
		M^{\prime}(0)X(s)=\left(a_3x_{2}(s)+a_2x_{3}(s),\,  a_3x_{1}(s)+(a_1+a_2)x_{3}(s),\,	a_2x_{1}(s)-(a_1+a_2)x_{2}(s)\right).
	\end{eqnarray*}
	Since $X(s)\times Y(s)=T(s)$, it follows that 
	$$\spn{M^{\prime}(0)X(s),Y(s)}=\spn{X(s)\times Y(s), (-a_1-a_2,-a_2,a_3)}=\spn{T(s),v}. $$
 Therefore, taking $v=(-a_1-a_2,-a_2,a_3)$ and $f^{\prime}(0)=c$, we have $k(s)=c+\spn{T(s),v}$.

	Conversely, let $X$ be a spacelike curve on $Q_{+}^{2}\subset \R_{1}^{3}$ parametrized by arc length $s$ such that $k(s)=c+\spn{T(s),v}$ for a vector $v \in \R^{3}_1\setminus\{0\}$ and $c \in \R$. Without loss of generality, up to isometries of $Q_{+}^{2}$, we can consider $v$ to be a multiple of $e_1=(1,0,0)$ if $v$ is a timelike vector, a multiple of $e_2=(1,1,0)$ if $v$ is a lightlike vector and a multiple of $e_3=(0,0,1)$ if $v$ is a spacelike vector. Thus, depending on the type of the vector $v$, the curvature  $k_i(s)=c+\langle T(s),v_i\rangle$ where $v_i=a\,e_i$, $a>0$ and $i=1,2,3$. 
	Now, we define the evolution of $X$ in $Q_{+}^{2}$ to be $\hat{X}_{i}(s,t)=f(t)M_{i}(t)X(s)$, where 
	\begin{eqnarray*} 
	M_{1}(t)&=&\left(\begin{array}{lll}
	1 & 0                    &                   0  \\
	0 & \cos(\varphi(t)) &  -\sin(\varphi(t))\\
	0 & \sin(\varphi(t)) &  \cos(\varphi(t))
	\end{array}\right),\,\,\nonumber
	M_{2}(t)=\left(\begin{array}{ccc}
	1+\frac{(\varphi(t))^{2}}{2} & -\frac{(\varphi(t))^{2}}{2}  &   -\varphi(t)  \\
	\frac{(\varphi(t))^{2}}{2} & 1-\frac{(\varphi(t))^{2}}{2}  &   -\varphi(t)  \\
	-\varphi(t)                    & \varphi(t)                     &   1 
	\end{array}\right),\label{1.1}\\
	M_{3}(t)&=&\left(\begin{array}{lll}
	\cosh(\varphi(t)) & \sinh(\varphi(t))  &   0  \\
	\sinh(\varphi(t)) & \cosh(\varphi(t))  &   0  \\
	0                    & 0                     &   1 
	\end{array}\right),\, \,\,\varphi(t)=\left\{\begin{array}{ll}
	\displaystyle	\frac{a}{c}\log(f(t)),\,\, \text{if}\,\,\, c \neq0,\\
	at,\,\,\text{if}\,\, c=0,\end{array}\right.
	\end{eqnarray*}
	and $f(t)=\sqrt{2ct+1}$. Taking the derivative with respect to $t$, it follows from \eqref{TYkhat} that 
	\begin{eqnarray*}
		\left\spn{\frac{\partial}{\partial t}\hat{X}_{i}(s,t),\hat{Y}_{i}(s,t)\right}&=&\frac{f^{\prime}(t)}{f(t)}+\left\spn{M^{\prime}_{i}(t)X(s),M_{i}(t)Y(s)\right}.
	\end{eqnarray*}	
	A straightforward computation shows that $\displaystyle \left\langle M^{\prime}_{i}(t)X(s),M_{i}(t)Y(s)\right\rangle=\varphi^{\prime}_i(t)\langle X(s)\times Y(s),w_{i}\rangle$. Thus
	\begin{eqnarray*}
		\left\spn{\frac{\partial}{\partial t}\hat{X}_{i}(s,t),\hat{Y}_{i}(s,t)\right}=\frac{f^{\prime}(t)}{f(t)}+ \varphi^{\prime}_{i}(t)\spn{T(s),w_{i}}
	\end{eqnarray*}
	for each $i=1,2,3$. If $c\neq 0$, then
	\begin{eqnarray*}
		\left\spn{\frac{\partial}{\partial t}\hat{X}_{i}(s,t),\hat{Y}_{i}(s,t)\right}
		=\frac{f^{\prime}(t)}{cf(t)}\left(c+a\spn{T(s),w_{i}}\right)
		=\frac{k_{i}(s)}{{f^{2}(t)}}).
	\end{eqnarray*}
It follows from \eqref{TYkhat} that $\displaystyle \left\spn{\frac{\partial}{\partial t}\hat{X}_{i}(s,t),\hat{Y}_{i}(s,t)\right}=\frac{k_{i}(s)}{{f^{2}(t)}})=\hat{k_{i}}(s,t)$. 
	
	When $c=0$, $k_{i}(s)=\spn{T(s),v_{i}}$, $f(t)\equiv 1$ and $\varphi_{i}(t)=at$. Hence the evolution is composed only by isometries and
	\begin{equation*}
		\left\spn{\frac{\partial}{\partial t}\hat{X}_{i}(s,t),\hat{Y}_{i}(s,t)\right}
		=a\spn{T(s),w_{i}}
		=k_{i}(s)
		=\hat{k}_{i}(s,t), 
	\end{equation*}
where the last equality follows from fact that isometries preserve curvature.

\end{proof}

The proof of  Theorem \ref{c3t3} is analogue to the proof of Theorem \ref{c3t2} and it will be omitted. It follows from Theorem \ref{c3t2} that the investigation  of the self-similar solutions to the CF in $Q_{+}^{2}$ reduces to studying the curves that satisfy equation (\ref{eq2}) for some vector $v \in \R^{3}_{1}\setminus \{0\}$ and $c \in \R$. Up to isometries of $Q_{+}^{2}$ and depending on the type of vector $v$,  we consider $v$ as being $v_{i}=ae_{i}$, where $a>0$, $e_1=(1,0,0)$, $e_2=(1,1,0)$ or 
$e_3=(0,0,1)$.
 Our next result characterizes \eqref{eq2} in terms of a system of differential equations.

\begin{proposition} \label{c4p1}
	Let $X: I \rightarrow Q^{2}_{+}$ be a spacelike curve parametrized by arc length $s$. Consider the vectors
	\begin{equation} \label{e_i}
	e_{1}=(1,0,0),\hspace{0,3 cm} e_{2}=(1,1,0), \hspace{0,3 cm} \text{e}\hspace{0,3 cm} e_{3}=(0,0,1).	
	\end{equation}
	For each $i \in \{1,2,3\}$, define the functions 
\begin{equation}\label{alphai}	
	\alpha_{i}(s)=\spn{X(s),e_{i}}, \qquad \tau_{i}(s)=\spn{T(s),e_{i}}, \qquad  \eta_{i}(s)=\spn{Y(s),e_{i}},
\end{equation}	
	 where $T$ is the unit tangent vector and $Y$ is the vector field associated to $X$. For fixed $a>0$ and $c \in \R$,
	\begin{equation*}
	k_{i}(s)=c+a\tau_{i}(s)
	\end{equation*}
	is satisfied for all $s \in I$, i.e. $X$ is a self-similar solution to the CF,  if and only if, the functions $\alpha_{i}(s)$, $\tau_{i}(s)$ and $\eta_{i}(s)$ satisfy the system
	\begin{equation}\label{c4e1}
	\left\{\begin{array}{lll}
	\alpha^{\prime}_{i}(s)=\tau_{i}(s),\\
	\tau^{\prime}_{i}(s)=[c+a\tau_{i}(s)]\alpha_{i}(s)-\eta_{i}(s),\\
	\eta^{\prime}_{i}(s)=-[c+a\tau_{i}(s)]\tau_{i}(s),
	\end{array}\right.
	\end{equation}
	with initial conditions $(\alpha_{i}(0),\tau_{i}(0),\eta_{i}(0))\neq (0,0,0)$ satisfying  
	\begin{equation}\label{e_i2}
	2\alpha_{i}(0)\eta_{i}(0)+\tau^{2}_{i}(0)=
	\left\{\begin{array}{lll}
	-1, \hspace{.3 cm} \text{se} \hspace{.3 cm} i=1,\\
	0, \hspace{.3 cm} \text{se} \hspace{.3 cm} i=2,\\
	1, \hspace{.3 cm} \text{se} \hspace{.3 cm} i=3.
	\end{array}\right.
	\end{equation}
For such a solution,  $2\alpha_i(s)\eta_i(s)+\tau^{2}_i(s)=-1$ (resp. 
$0$ or $1$) when $i=1$ (resp. $2$ or $3$), for all $s$.
\end{proposition}

\begin{proof}
	The vector fields $X$, $T$ and $Y$ satisfy the following equations (see \cite{Liu})
	\begin{equation}\label{c4e3}
	\left\{\begin{array}{lll}
	X^{\prime}(s)=T(s),\\
	T^{\prime}(s)=k(s)X(s)-Y(s),\\
	Y^{\prime}(s)=-k(s)T(s),
	\end{array}\right.
	\end{equation}
	for all $s \in I$. Thus, taking the inner product with $e_i$, we obtain 
	\begin{equation}
	\left\{\begin{array}{lll}\label{c4e2}
	\alpha^{\prime}_{i}(s)=\tau_{i}(s),\\
	\tau^{\prime}_{i}(s)=k_{i}(s)\alpha_{i}(s)-\eta_{i}(s),\\
	\eta^{\prime}_{i}(s)=-k_{i}(s)\tau_{i}(s).
	\end{array}\right.
	\end{equation}
Assume that $k_{i}(s)=c+a\tau_{i}(s)$ for all $s \in I$, then  \eqref{c4e1} is satisfied. Moreover, since $\{X(s),T(s),Y(s)\}$ is a basis for $\R_{1}^{3}$ for each $s$, it follows from \eqref{alphai} that  $e_i=\eta_i(s)X(s)+\tau_i(s)T(s)+\alpha_i(s)Y(s)$ i.e. $\spn{e_i,e_i}=2\alpha_{i}(s)\eta_{i}(s)+\tau^{2}_{i}(s)$ for all $s \in I$. In particular, at $s=0$ we have that \eqref{e_i2} holds.
	
	Conversely, suppose that the functions $\alpha_{i}(s)$, $\tau_{i}(s)$ and $\eta_{i}(s)$ satisfying \eqref{c4e1} and \eqref{e_i2} for each $i \in \{1,2,3\}$. Since \eqref{c4e2} holds, then $\left[c+a\tau_{i}(s)-k_{i}(s)\right]\alpha_{i}(s)=0$ and 
	$\left[c+a\tau_{i}(s)-k_{i}(s)\right]\tau_{i}(s)=0$, 
	for all $s \in I$. Assume by contradiction that $c+a\tau_{i}(s)-k_{i}(s)\neq 0$ on some interval $J \subset I$. Then $\alpha_{i}(s)=\spn{X(s),e_{i}}=0$, $\tau_{i}(s)=\spn{T(s),e_{i}}=0$ for all $s \in J$  and hence  $\spn{T'(s),e_{i}}=\spn{Y(s),e_i}=0$, which is a contradiction since $e_i\neq 0$. Therefore, $k_{i}(s)=c+a\tau_{i}(s)$ for all $s \in I$ and  $i \in \{1,2,3\}$.
\end{proof}
Our next proposition shows that a solution of the  system \eqref{c4e1}, with initial conditions satisfying \eqref{e_i2}, provides a self-similar solution to the CF in $Q_{+}^{2}$. 

\begin{proposition} \label{c4p2}
	Given a solution $\alpha(s), \tau(s), \eta(s): I\rightarrow \R$ of the system \eqref{c4e1} on some interval $J$ for fixed $a>0$, $c \in \R$ and initial conditions $(\alpha(0),\tau(0),\eta(0))\neq (0,0,0)$ satisfying $2\alpha(0)\eta(0)+\tau^{2}(0)=-1$ (resp. $0$ and $1$), there exists  smooth  spacelike curve $X: I \rightarrow Q_{+}^{2}$, parametrized by arc length $s$,  which is a self-similar solution to the CF, with curvature $k(s)=c+a\tau(s)$,  such that its tangent and associated vector fields $T$ and $Y$ satisfy
	\begin{equation}\label{c4e5}
	\alpha(s)=\spn{X(s),e},\hspace{.5 cm}\tau(s)=\spn{T(s),e} \hspace{.5 cm} \text{\rm e} \hspace{.5 cm} \eta(s)=\spn{Y(s),e}, 
	\end{equation} 
	where $e=(1,0,0)$ (resp. $e=(1,1,0)$ \rm and \it $e=(0,0,1)$). 
\end{proposition}
\begin{proof} Given a solution of \eqref{c4e1}, 
	define $k(s)=c+a\tau(s)$. Up to isometries of $Q_{+}^{2}$, there exists a unique curve $\displaystyle X: I \rightarrow Q_{+}^{2}$, whose curvature is $k(s)$ i.e. $X(s)$ and its tangent and associated vector fields $T(s)$ and $Y(s)$ satisfy \eqref{c4e3}.  The curve $X(s)$ is uniquely determined by the initial conditions $X(0)$, $T(0)$ and $Y(0)$, that can be chosen such that $\eta(0)X(0)+\tau(0)T(0)+\alpha(0)Y(0)=e,$ where $e=(1,0,0)$ (resp. $e=(1,1,0)$ and $e=(0,0,1)$).    A straightforward computation shows that \eqref{c4e1} and \eqref{c4e3} imply
	$	\frac{d}{ds}\left(\eta(s)X(s)+\tau(s)T(s)+\alpha(s)Y(s)\right)=0$ for all $s$. Therefore, (\ref{c4e5}) is satisfied and  Theorem \ref{c3t2} implies that $X$ is a self-similar solution to the CF.
\end{proof}
 Propositions \ref{c4p1} and \ref{c4p2}, show that investigating the  self-similar solutions to the CF on $Q_{+}^{2}$  is equivalent to studying the  solutions $\psi(s)=(\alpha(s),\tau(s), \eta(s))$ of the system 
\begin{equation}\label{s2}
\left\{\begin{array}{lll}
\alpha^{\prime}(s)=\tau(s),\\
\tau^{\prime}(s)=[c+a\tau(s)]\alpha(s)-\eta(s),\\
\eta^{\prime}(s)=-[c+a\tau(s)]\tau(s),
\end{array}\right.
\end{equation}
for given constants $c \in \R$, $a>0$ and initial condition $\psi(0) \in H\cup C\cup S$, where
\begin{equation}\label{h} \begin{array}{l}
H:=\{(\alpha,\tau,\eta) \in \R^{3}: 2\alpha\eta+\tau^{2}=-1\,\, \text{and} \,\, \alpha<0\},\\	
C:=\{(\alpha,\tau,\eta) \in \R^{3}\setminus \{0\}: 2\alpha\eta+\tau^{2}=0\,\, \text{and} \,\,\alpha \leq 0\}, \\
S:=\{(\alpha,\tau,\eta) \in \R^{3}: 2\alpha\eta+\tau^{2}=1\}. 
\end{array}	
\end{equation}
These are disjoint sets and if the initial condition $\psi(0) \in H$ (resp. $C$ or $S$) then the solution $\psi(s)$,  defined on the maximal interval $I$, will be contained in $H$ (resp. $C$ or $S$).

\begin{remark}\label{ob1} \rm
	Let $\displaystyle X: I \rightarrow Q_{+}^{2}$ be a spacelike curve parametrized by arc length $s$ given by $X(s)=(x_1(s),x_2(s), x_3(s))$ and let $Y(s)=(y_1(s),y_2(s),y_3(s))$ be the vector field  associated to $X(s)$. The function $\alpha(s)$ defined by \eqref{c4e5} has the following geometric interpretation.
	\begin{itemize}
		\item If $e=(1,0,0)$ (timelike vector), then $\alpha(s)=-x_1(s)<0$ and $\eta(s)=-y_{1}(s)>0$ for all $s \in I$. Moreover, $\alpha(s)$ is the Euclidean height function with respect to the vector $(-1,0,0).$
		\item If $e=(1,1,0)$ (lightlike vector), then $\alpha(s)=-x_1(s)+x_2(s)\leq 0$ and $\eta(s)=-y_1(s)+y_2(s)\geq 0$ for all $s \in I$. Moreover, $\alpha(s)/\sqrt{2}$ is the Euclidean height function with respect to the vector $(-1/\sqrt{2},1/\sqrt{2},0).$
		\item If $e=(0,0,1)$ (spacelike vector), then $\alpha(s)=x_3(s)$ for all $s \in I$. Moreover, $\alpha(s)$ is the Euclidean height function (with sign) with respect to the vector $(0,0,1).$
	\end{itemize}
\end{remark}

We can prove results similar to  Propositions \ref{c4p1} and \ref{c4p2} for the ICF. In fact, using the same  arguments, we have the following

\begin{proposition}\label{propinvers} A spacelike curve parametrized by  arc length $s$, 
$\tilde{X}:I \rightarrow Q_{+}^{2}$, $s \in I$, whose curvature $\tilde{k}(s)\neq 0$, 
is a self-similar solution to the ICF, i.e. $c+\spn{\tilde{T},v}=1/\tilde{k}$ for a vector $v\in\R^3_1\setminus\{0\}$ and $c\in \R$ if, and only if, $\tilde{X}$   
is determined by a solution $\tilde{\psi}(s)=(\tilde{\alpha}(s),\tilde{\tau}(s), \tilde{\eta}(s))$ defined on the maximal interval $I$ of the system
\begin{equation}\label{s3}
\left\{\begin{array}{lll}
\displaystyle \tilde{\alpha}^{\prime}(s)=\tilde{\tau}(s),\\
\displaystyle \tilde{\tau}^{\prime}(s)=\frac{\tilde{\alpha}(s)}{c+a\tilde{\tau}(s)}-\tilde{\eta}(s),\\
\displaystyle \tilde{\eta}^{\prime}(s)=-\frac{\tilde{\tau}(s)}{c+a\tilde{\tau}(s)},
\end{array}\right.
\end{equation}
for a  constant $a>0$ and initial condition $\psi(0) \in H\cup C\cup S$, where $\tilde{\alpha}(s)=\spn{\tilde{X}(s),e}$, $\tilde{\tau}(s)=\spn{\tilde{T}(s),e}$, $\tilde{\eta}(s)=\spn{\tilde{T}(s),e}$ , $e \in \{(1,0,0), (1,1,0), (0,0,1)\}$ and $v=ae$.
\end{proposition}

From now on, unless explicitly stated, we will restrict ourselves to studying the self-similar solutions to the CF on $Q^2_+$.  The proof of  Theorem \ref{c3t4} will follow from a long series of lemmas on the properties of the 
solutions of  the system of differential equations \eqref{s2}.  Our first lemma will be used repeatedly and it  provides two relations  between the functions  $\alpha(s)$, $\tau(s)$ and $\eta(s)$, which depend on the initial condition.  
\begin{lemma}\label{novo}
	Let $\psi(s)=(\alpha(s),\tau(s),\eta(s))$, $s \in I \subset \R$ be a solution of \eqref{s2}, $a>0$, $c \in \R$ and initial condition $\psi(0) \in H \cup C \cup S$. Then for all $s \in I$
	\begin{equation}\label{novo1}
	[c+a\tau(s)]\alpha^{2}(s)-\alpha(s)\tau^{\prime}(s)+\frac{\tau^{2}(s)}{2}=\displaystyle \left\{\begin{array}{ccc}
	\displaystyle	-\frac{1}{2} & \text{if} & \psi(0) \in H,\\
	\displaystyle	0 & \text{if} &\psi(0) \in C,\\
	\displaystyle	\frac{1}{2} & \text{if} &\psi(0) \in S		
	\end{array}\right.
	\end{equation}
	and \begin{equation}\label{novo2}
	\eta^{2}(s)+\eta(s)\tau^{\prime}(s)+\frac{[c+a\tau(s)]\tau^{2}(s)}{2}=\displaystyle{\left\{\begin{array}{lll}
		\displaystyle			-\,\frac{[c+a\tau(s)]}{2} & \text{if}& \psi(0) \in H,\\
		0 & \text{if} &\psi(0) \in C,\\
		\displaystyle		\frac{[c+a\tau(s)]}{2} & \text{if} &\psi(0) \in S.		
		\end{array}\right.}
	\end{equation}
\end{lemma}

\begin{proof}
	Multiplying the second equation of \eqref{s2} by $\alpha(s)$ (resp. $\eta(s)$) and using the fact that $2\alpha(s)\eta(s)+\tau^{2}(s)=-1$ (resp. $0,1$) in $H$ (resp. $C,$ $S$), we obtain \eqref{novo1} (resp. \eqref{novo2}).
\end{proof}
In our next lemma, we study the solutions of \eqref{s2} with constant function $\tau(s)$. 
\begin{lemma}\label{tri}
	Let $\psi(s)=(\alpha(s),\tau(s),\eta(s))$ be a non null solution of (\ref{s2}) defined on the maximal interval $I$, $a>0$, $c \in \R$ and initial condition $\psi(0)\in H\cup C\cup S$. The function $\tau(s)=b$, $\forall s \in I$, where $b\in\R$  if, and only if, $b \in \{-1,0,1\}$ and $I=\R$. Moreover,
	\begin{itemize}
		\item[i)] if $b=0$, then $\psi(s)=(\alpha_0,0,c\alpha_0)$ is a  singular solution of \eqref{s2} and $\psi(s) \in H$ (resp. $C$, $S$), for all $s\in\R$, when $c<0$ (resp. $c=0$, $c>0$);
		\item[ii)] if $b^{2}=1$, then $a^{2}=c^{2}$ and $\psi(s)=(\pm s+\alpha_0,\pm1,0) \in S$ for all $s \in \R.$
	\end{itemize}
\end{lemma}
\begin{proof} 
 It folows from \eqref{s2} that $\tau(s)=b$, if and only if $\alpha(s)=bs+\alpha_0$, $\eta(s)=(c+ab)(bs+\alpha_0)$ and $2(c+ab)b=0$. 
Since $2\alpha(s)\eta(s)+\tau^{2}(s)=\gamma$, where $\gamma \in \{-1,0,1\}$ then if $b=0$, then  i) holds and if $b\neq 0$ then $c+ab=0$,  $\eta=0$, 
$b^2=1$ and $a^2=c^2$, which proves ii).     
\end{proof}

Lemma \ref{tri} shows that the solutions $\psi(s)$ of (\ref{s2}), for which  $\tau(s)$ is  constant,  are explicitly given and they  correspond to the conic sections of $Q_{+}^{2}$, with constant curvature $k(s)=c+ab$.  
In this context, we define a \it trivial solution \rm $\psi(s)=(\alpha(s),\tau(s),\eta(s))$ of \eqref{s2} when $\tau(s)$ is a constant function. From now on, we will study only non trivial solutions $\psi(s)$.

Our next three lemmas analyze the critical points of the functions $\alpha(s)$, $\tau(s)$ and $\eta(s)$. 

\begin{lemma}\label{ln1}
	Let $\psi(s)=(\alpha(s),\tau(s),\eta(s))$ be a non trivial solution of (\ref{s2}) defined on the maximal interval $I$, $a>0$, $c \in \R$ and initial condition $\psi(0)\in H\cup C\cup S$. If $s_0$ is a critical point of $\alpha(s)$ then it is a local minimum (resp. maximum) point of $\alpha(s)$ if, and only if, $c\alpha(s_0)-\eta(s_0)>0$ (resp. $c\alpha(s_0)-\eta(s_0)<0$). Moreover: 
	\begin{itemize}
		\item[i)] if $c\geq 0$ and $\psi(0)\in H\cup C$, then $s_0$ is a global maximum point of $\alpha(s)$;
		\item[ii)] if $c< 0$ and $\psi(0)\in H$, then $s_0$ is a local minimum (resp. maximum) point of $\alpha(s)$ if, and only if, $2c\alpha^{2}(s_0)+1<0$ (resp. $2c\alpha^{2}(s_0)+1>0$);
		\item[iii)] if $c< 0$ and $\psi(0))\in C$, then $s_0$ is a local minimum (resp. maximum) point of $\alpha(s)$ if, and only if, $\alpha(s_{0})<0$ (resp. $\alpha(s_{0})=0$).
	\end{itemize}
	\end{lemma}
\begin{proof}	
	Let $s_0$ be a critical point of $\alpha(s)$. It follows from \eqref{s2} that $\alpha^{\prime \prime}(s_0)=\tau^{\prime}(s_0)=c\alpha(s_0)-\eta(s_0).$
	If there exists $s_0 \in I$ such that $\tau(s_0)=\tau^{\prime}(s_0)=\alpha^{\prime \prime}(s_0)=0$, then $(\alpha(s_0),0,c\alpha(s_0))$ is a singular solution of \eqref{s2}, which contradicts the hypothesis. Thus, $\alpha^{\prime \prime}(s_0)\neq 0$ and the result holds. 
	
	\it i) \rm When $\psi(0)\in H\cup C$ and $c\geq0$ we have  $\alpha(s_0)\leq 0$ and $\eta(s_0)\geq 0$ for all $s \in I $ and hence $\alpha^{\prime \prime}(s_0)=c\alpha(s_0)-\eta(s_0)<0$ i.e. $s_0$ is a global maximum point of $\alpha(s).$ 
	
	\it ii) \rm If $\psi(0)\in H$ and $c<0$ then $\alpha(s_0)<0$, $\eta(s_0)>0$, $2\alpha(s_0)\eta(s_0)=-1$ and $2\alpha(s_0)\alpha^{\prime \prime}(s_0)=2c\alpha^{2}(s_0)+1.$
	
	\it iii) \rm If $\psi(0)\in C$ and $c<0$, then $\alpha(s_0)\leq0$, $\eta(s_0)\geq 0$, $2\alpha(s_0)\eta(s_0)=0$, $\alpha^{2}(s_0)+\eta^{2}(s_0)\neq 0$ and $\alpha^{\prime \prime}(s_0)=c\alpha(s_0)-\eta(s_0).$  This proves item \it iii). \rm	 
\end{proof}

\begin{lemma}\label{ln2}
	Let $\psi(s)=(\alpha(s),\tau(s),\eta(s))$ be a non trivial solution of (\ref{s2}) defined on the maximal interval $I$, $a>0$, $c \in \R$ and initial condition $\psi(0)\in H\cup C\cup S$. Let $s_0$ be a critical point of $\eta(s).$ If $c=0$ then $\eta(s)$ is a  decreasing function and $s_0$ is an inflection point of $\eta(s)$.  If $c\neq0$ then $[c+a\tau(s_0)]\tau(s_0)=0.$ In this case
	\begin{itemize}
		\item[i)] If $\tau(s_0)=0$ then $\eta^{\prime \prime}(s_0)\neq 0$ and $s_0$  is a local minimum (resp. maximum) point of $\eta(s)$ if, and only if, $-c^{2}\alpha(s_0)+c\eta(s_0)>0$ (resp. $-c^{2}\alpha(s_0)+c\eta(s_0)<0$).
		\item[ii)] If $c+a\tau(s_0)=0$ then $\eta^{\prime \prime}(s_0)\neq 0$ and $s_0$ is a local minimum (resp. maximum) point of $\eta(s)$ if, and only if, $c\eta(s_0)<0$ (resp. $c\eta(s_0)>0$).
	\end{itemize}
\end{lemma}
\begin{proof}
	Let $s_0$ be a critical point of $\eta(s)$, then $[c+a\tau(s_0)]\tau(s_0)=0$.  When $c=0$ it follows from \eqref{s2} that $\eta(s)$ is a decreasing function on $I$.
	
	\it i) \rm Suppose that $c\neq 0$ and $\tau(s_0)=0$, then (\ref{s2}) implies that $\eta^{\prime \prime}(s_{0})=-c^{2}\alpha(s_{0})+c\eta(s_0)$.  Moreover,   $\eta^{\prime \prime}(s_0)\neq 0$, otherwise,    $\psi(s)=(\alpha(s_0),0,\eta(s_0))$, $s \in \R$ would be a singular solution of \eqref{s2}, which contradicts the hypothesis. 
	
	\it ii) \rm Suppose $c\neq 0$ and $\displaystyle c+a\tau(s_{0})=0$,    then  (\ref{s2}) implies  that 	
	$\eta^{\prime\prime}(s_{0})=c\tau^{\prime}(s_{0})=-c\eta(s_{0})$. 
	Moreover, $\eta^{\prime \prime}(s_0)\neq 0$. In fact, otherwise  $\alpha(s_0)\eta(s_0)=0,$ $\displaystyle a^2\tau^{2}(s_0)=c^{2}$. Thus, $\psi(s_0) \in S$,  $c^{2}=a^{2}$, $\tau^{2}(s_0)=1$ and $\psi(s_0)=(\alpha(s_0),\pm1,0)$. However, it follows from Lemma \ref{tri} that $\psi(s)=(\pm s+\alpha(s_0),\pm1,0)$, $s \in \R$ are trivial solutions of \eqref{s2},  contradicting the hypothesis. 
\end{proof}

\begin{lemma}\label{ln3}
	Let $\psi(s)=(\alpha(s),\tau(s),\eta(s))$ be a non trivial solution of (\ref{s2}) defined on the maximal interval $I$, $a>0$, $c \in \R$ and initial condition $\psi(0)\in H\cup C\cup S$. If $s_0$ is a critical point of $\tau(s)$ then $\tau^{\prime \prime}(s_0)\neq 0$ and it is a local minimum (resp. maximum) point of $\tau(s)$ if, and only if, $\displaystyle \tau(s_0)[c+a\tau(s_0)]>0$ (resp. $\displaystyle \tau(s_0)[c+a\tau(s_0)]<0$). In particular:
	\begin{itemize}
		\item[i)] if $c=0$ then $s_0$ is a global minimum point of $\tau(s);$
		\item[ii)] if $\psi(0)\in H\cup C$ then $c+a\tau(s_0)<0;$
		\item[iii)] if $c>0$ and $\psi(0)\in H\cup C$ then $s_0$ is a global minimum point of $\tau(s)$.
	\end{itemize}	
\end{lemma}
\begin{proof}
	Let $s_{0}$ be a critical point of $\tau(s)$, then  $[c+a\tau(s_{0})]\alpha(s_{0})=\eta(s_{0})$. Taking the second derivative of $\tau(s)$ at $s_0$, it follows from\eqref{s2}, that 		$\tau^{\prime \prime}(s_{0})
		=2[c+a\tau(s_{0})]\tau(s_{0})$.
Assume by contradiction that $\tau^{\prime \prime}(s_{0})=0$, i.e.  
$2[c+a\tau(s_{0})]\tau(s_{0})=0$,  
then  either $\tau(s_{0})=0$ or $c+a\tau(s_0)=0$. If $\tau(s_0)=0$ then $c\alpha(s_0)=\eta(s_0)$ and $\psi(s)=(\alpha(s_0),0,c\alpha(s_0))$, $s \in \R$ would be a singular solution of \eqref{s2}, which contradicts the hypothesis. If $c+a\tau(s_0)=0$ we would have $\eta(s_0)=0$, $a^2\tau^{2}(s_0)=c^{2}.$ Thus, $\psi(s_0) \in S$, $c^{2}=a^{2}$, $\tau^{2}(s_0)=1$ and $\psi(s_0)=(\alpha(s_0),\pm1,0)$.  However, it follows from Lemma \ref{tri} that $\psi(s)=(\pm s+\alpha(s_0),\pm1,0)$, $s \in \R$ are trivial solutions of \eqref{s2}, contradicting the hypothesis.

\it i) \rm Suppose that $c=0$ then $\tau^{\prime \prime}(s_0)=a\tau^{2}(s_0)>0$ i.e. $s_0$ is a global minimum point of $\tau(s)$. \

\it ii) \rm If $\psi(0)\in H\cup C$ then $\alpha(s_0)<0$ and $\eta(s_0)>0$. Thus,  $[c+a\tau(s_{0})]\alpha(s_{0})=\eta(s_{0})>0$ and $c+a\tau(s_0)<0.$

\it iii) \rm If $c>0$ and $\psi(0)\in H\cup C$ then it follows from \it ii) \rm that $c+\tau(s_0)<0$ and $\tau(s_0)<0$. Therefore, $\tau^{\prime \prime}(s_0)> 0$ and $s_0$ is a global minimum point of $\tau(s).$
\end{proof}

 Our next lemma shows that the curvature function $k(s)=c+a\tau(s)$ of a self-similar solution to the CF has at  most two zeros. The importance of the number of zeros of the curvature is due to the fact that it determines the number of connected components of the associated self-similar solutions to the ICF.  
  
\begin{lemma}\label{ln4}
	Let $\psi(s)=(\alpha(s),\tau(s),\eta(s))$ be a non trivial solution of (\ref{s2}) defined on the maximal interval $I=(\omega_{-}, \omega_{+})$, $a>0$, $c \in \R$ and initial condition $\psi(0)\in H\cup C\cup S$.
\begin{itemize}
		\item[i)] If $c\geq0$ then $\tau(s)$ has at  most two zeros.
		\item[ii)] If $c<0$, then $\tau(s)$ has an infinite number  of zeros in $ I$ if, and only if, the functions $\alpha(s)$, $\tau(s)$ and $\eta(s)$ have an infinite number  of critical point in $I$.
		\item[iii)] The function $c+a\tau(s)$ has at  most two zeros.
	\end{itemize}
\end{lemma}

\begin{proof}
	\it i) \rm If  $c\geq0$ and $\tau({s})>0$, for $s\in J\subset I$, then  $[c+a\tau(s)]\tau(s)>0$, and it follows from Lemma \ref{ln3} that $\tau(s)$ has no points of local maximum in $J$. Moreover, if $s_0 \in I$ satisfies $\tau(s_0)=0$ then $s_0$ is not a critical point of $\tau$. Hence,  if $\tau^{\prime}(s_0)>0$ then $\tau(s)$ is increasing on $ (s_0,\omega_+)$ and if $\tau^{\prime}(s_0)<0$ then $\tau(s)$ is decreasing on$ (\omega_-,s_0)$. Therefore, the function $\tau(s)$ has at  most two zeros. 
		
	\it ii) \rm It follows from \eqref{s2} that the zeros of $\tau(s)$ are critical points of $\alpha(s)$ and $\eta(s)$. Suppose that $c<0$ then  Lemma \ref{ln3} implies that $\tau(s)$ may have local maximum (resp. minimum) points when $0<a\tau(s)<-c$ (resp. $\tau(s)<0$). 
	 Therefore, if $c<0$ then $\tau(s)$ has an infinite number of zeros if, and only if, $\tau(s)$ has an infinite number of local critical points.
	
	\it iii) \rm When $c\leq 0$ and $c+a\tau(s)>0$ for $s\in J\subset I$, then  $[c+a\tau(s)]\tau(s)>0$ and  it follows from Lemma \ref{ln3} that  $\tau(s)$ has no local maximum in $J$. Moreover, if $s_0 \in I$ is such that $c+a\tau(s_0)=0$ then $s_0$ is not a critical point of $\tau$. 
	Therefore,  when $c\leq 0$, the function $c+a\tau(s)$ has at  most two zeros. 
	
	Suppose that $c>0$ and $\psi(0)\in H\cup C$, it follows from item \it iii) \rm of Lemma \ref{ln3}  that $\tau(s)$ has at  most one critical point, thus, $c+a\tau(s)$ has at  most two zeros.
		 Finally, suppose that $c>0$ and $\psi(0)\in S$ and assume by contradiction that there exist three consecutive zeros  $s_1,s_2,s_3 \in I$ such that $s_1<s_2<s_3$, $c+a\tau(s_i)=0$, for $i=1,23$ and $c+a\tau(s)\neq 0$ for all $s \in (s_1,s_2)\cup (s_2,s_3)$. It follows from Lemma \ref{ln3} that $\tau(s)<0$ for all $s \in (s_1,s_3)$, because $\tau(s)$ has no local maximum points when $\tau(s)>0$. Moreover,  $s_1,s_2$ and $s_3$ are also critical points of $\eta(s)$.  If $s_1$ is a local maximum of $\eta(s)$, then it follows from (\ref{s2}) and Lemma \ref{ln2} that: $\eta(s_1)>0$, $\tau(s)$ is decreasing at $s=s_1$, $s_2$ is a local minimum of $\eta(s)$, $\eta(s_2)<0$, $\displaystyle c+a\tau(s)<0$ for all $s \in (s_1,s_2)$, $s_3$ is a local maximum of $\eta(s)$, $\eta(s_3)>0$ and $c+a\tau(s)>0$ for all $s \in (s_2,s_3)$. Thus, there exist $ b \in (s_1,s_2)$ and $d \in (s_2,s_3)$ such that $\eta(b)=\eta(d)=0$. Since $2\alpha(s)\eta(s)+ \tau^{2}(s)=1$ and  $\tau(s)<0$ for all $s \in (s_1,s_3)$, it follows that $\tau(b)=\tau(d)=-1$ i.e. $\displaystyle c-a<0$ and $\displaystyle c-a>0$. This is a contradiction. With similar arguments we obtain a contradiction if $s_1$ is a local minimum of $\eta(s)$.  Therefore, $c+a\tau(s)$ has at  most two zeros.	
\end{proof}

\begin{lemma}\label{ln5}
	Let $\psi(s)=(\alpha(s),\tau(s),\eta(s))$ be a non trivial solution of (\ref{s2}) defined on the maximal interval $I=(\omega_{-}, \omega_{+})$, $a>0$, $c<0$ and initial condition $\psi(0)\in H\cup C\cup S$. If $\tau(s)$ has an infinite number of zeros in the interval $(\omega_{-},\overline{s})$ (resp. $(\overline{s},\omega_{+})$) for some $\overline{s} \in I$ fixed, then the functions  $\alpha(s),$ $\tau(s)$ and $\eta(s)$ are bounded on the intervals $(\omega_{-},\overline{s})$ (resp. $(\overline{s},\omega_{+})$). 
\end{lemma}

\begin{proof} 
	Note that, the function $g(s)=c\alpha(s)+\eta(s)$, $s \in I$ is decreasing on $I$, since $g^{\prime}(s)=-a\tau^{2}(s)$ for all $s \in I.$ When $c<0$ it follows from item \it ii) \rm of  Lemma \ref{ln4} that the functions $\alpha(s)$, $\tau(s)$ and $\eta(s)$ have an infinite number of critical points if $\tau(s)$ has an infinite number of zeros. Thus, suppose that $ (s_{k})_{k \in \N}\subset (\omega_{-},\overline{s})$ such that $\tau(s_{k})=0$, $k \in \N$, $\tau(s)\neq0$ for all $s \in (s_{k+1},s_k)$ and $ \lim_{k \to +\infty}s_{k}=\omega_{-}$. Without loss generality, we can consider: $ (s_{2k})_{k \in \N}$ the local maximum points of $\alpha(s)$; $ (s_{2k+1})_{ \in \N}$ the local minimum points of $\alpha(s)$;  $\tilde{s}_{2k} \in (s_{2k+1},s_{2k})$ the local maximum points of $\tau(s)$ and $\tilde{s}_{2k+1} \in (s_{2k+2},s_{2k+1})$ the local minimum points of $\tau(s).$
	 Hence $\tau(s)>0$ on $(s_{2k+1},s_{2k})$ and $\tau(s)<0$ on $(s_{2k+2},s_{2k+1})$. It follows from Lemma \ref{ln3} that the local maximum points of $\tau(s)$ satisfy $\tau(\tilde{s}_{2k})>0$ and $c+ a\tau(\tilde{s}_{2k})<0$ for all $k \in \N$, and $ a\tau(s)+c<0$ for all $s \in (s_{2k+1},s_{2k}).$ In the similar way, suppose that $\displaystyle (s_{j})_{j \in \N}\subset (\overline{s},\omega_{+})$ such that $\tau(s_{j})=0$, $j \in \N$, $\tau(s)\neq0$ for all $s \in (s_{j},s_{j+1})$ and $ \lim_{j \to +\infty}s_{j}=\omega_{+}$.  Without loss generality, we can consider: $\displaystyle (s_{2j})_{j \in \N}$ the local maximum points of $\alpha(s)$; $\displaystyle (s_{2j+1})_{j \in \N}$ the local minimum points of $\alpha(s)$; $\tilde{s}_{2j} \in (s_{2j+1},s_{2j+2})$ the local maximum points of $\tau(s)$ and $\tilde{s}_{2j+1} \in (s_{2j+2},s_{2j+3})$ the local minimum points of $\tau(s).$ Hence $\tau(s)>0$ on $(s_{2j+1},s_{2j+2})$ and $\tau(s)<0$ on $(s_{2j+2},s_{2j+3})$. Thus, it follows from Lemma \ref{ln3} that the local maximum points of $\tau(s)$ satisfy $c+a\tau(\tilde{s}_{2j})<0$, $k \in \N$ and $ c+a\tau(s)<0$ for all $s \in (s_{2j+1},s_{2j+2}).$ We will divide the proof in two cases: $\psi(0)\in H\cup C$ and $\psi(0)\in S.$
	
	When $\psi(0)\in H\cup C$ then $\alpha(s)\leq 0$ and $\eta(s)\geq0$ for all $s \in I$ i.e. the decreasing function $g(s)=c\alpha(s)+\eta(s)$ is non negative and bounded on the interval $(\overline{s},\omega_{+})$. Moreover, $0\leq c\alpha(s)\leq g(s)\leq g(\overline{s})$, $0\leq\eta(s)\leq g(s)\leq g(\overline{s})$ for all $s>\overline{s},$ i.e. the functions $\alpha(s)$ and $\eta(s)$ are bounded on $(\overline{s},\omega_{+})$ and it follows from $2\alpha(s)\eta(s)+\tau^{2}(s)=\delta$, where $\delta \in \{-1,0\}$ that $\tau(s)$ is also bounded.
	
	We claim that $\alpha(s)$ is bounded on the interval $(\omega_{-},\overline{s}).$ In fact, assume by contradiction that $\alpha(s)$ is unbounded i.e. the sequence of the local minimum value of $\alpha(s)$, $ (\alpha(s_{2k+1}))_{k \in \N}\subset (\omega_{-},\overline{s})$ is unbounded.   When $\psi(0) \in H,$ then it follows from item \it ii) \rm of  Lemma \ref{ln1} that $\sqrt{-2c}\alpha(s_{2k+1})<-1$ and $  \sqrt{-2c}\alpha(s_{2k})>-1$ for all $k \in \N$ i.e. there exists $t_{2k}\in  (s_{2k+1},s_{2k})$ such that $\sqrt{-2c}\alpha(t_{2k})=-1$ for all $k \in \N$. When $\psi(0) \in C,$ it follows from item \it iii) \rm of  Lemma \ref{ln1} that $ \alpha(s_{2k+1})<0$ and $ \displaystyle \alpha(s_{2k})=0$ for all $k \in \N$, moreover, $\displaystyle \alpha(s_{2k+3})<\alpha(s_{2k+1})$ for all $k\in \N$, because $g(s)=c\alpha(s)+\eta(s)$ is a decreasing function and $\eta(s_{2k+1})=\eta(s_{2k+3})=0$, $k \in \N.$ Thus, there exist $k_{0}\in \N$ and $t_{2k}\in  (s_{2k+1},s_{2k})$ satisfying $\displaystyle \sqrt{-2c}\alpha(t_{2k})=-1$ for all $k >k_0$. Since $ c+a\tau(s)<0$ and $\tau(s)>0$ on $(s_{2k+1},s_{2k})$, and $2\alpha(s)\eta(s)+\tau^{2}(s)=\delta$, where $\delta \in \{-1,0\}$, then $\displaystyle (\eta(t_{2k}))_{k >k_0}$ is also bounded. Hence, $\displaystyle (g(t_{2k}))_{k >k_0}$ is bounded and monotone. Therefore, $g(s)$ is bounded and $0 \leq c\alpha(s)\leq g(s)$ for all $s \in (\omega_{-},\overline{s})$.  But this  contradicts the assumption that $\alpha(s)$ is unbounded. Thus, $\alpha(s)$ is bounded on the interval $(\omega_{-},\overline{s}).$ 
	
	It follows from \eqref{novo2} that 
	\begin{equation*}
	\eta^{2}(\tilde{s}_{2k})=\displaystyle{\left\{\displaystyle{\begin{array}{lll}
			\displaystyle -\frac{[c+a\tau(\tilde{s}_{2k})][1+\tau^{2}(\tilde{s}_{2k})]}{2} & \text{if}& \psi(0) \in H,\\ \\
			\displaystyle	-\frac{[c+a\tau(\tilde{s}_{2k})]\tau^{2}(\tilde{s}_{2k})}{2} & \text{if} &\psi(0) \in C.
			\end{array}}\right.}
	\end{equation*}
	Since $\tau(\tilde{s}_{2k})>0$ and $c+ a\tau(\tilde{s}_{2k})<0$, then $(\eta(\tilde{s}_{2k}))_{k \in \N}$ is bounded. Therefore, $g(\tilde{s}_{2k})=c\alpha(\tilde{s}_{2k})+\eta(\tilde{s}_{2k})$, $k \in \N$ is bounded and monotone. Thus, $g(s)$ is bounded, $0\leq c\alpha(s)\leq g(s)$, $0 \leq \eta(s)\leq g(s)$ for all  $s \in (\omega_{-},\overline{s})$. Therefore, the functions $\alpha(s)$, $\eta(s)$ and $\tau(s)$ are bounded on the interval $(\omega_{-},\overline{s}).$
	
	Finally, we will study the case $\psi(0)\in S$. Note that, when $s_0$ satisfy $\tau(s_0)=0$, then $s_0$ is a critical point of $\alpha(s)$ and of  $\eta(s)$, and $2\alpha(s_0)\eta(s_0)=1$. Thus, it follows from Lemmas \ref{ln1} and \ref{ln2} that, if $\alpha(s_0)>0$ (resp. $\alpha(s_0)<0$) then $s_0$ is a local maximum (resp. minimum) point of $\alpha(s)$ and of $\eta(s)$
	
	\bf Claim: \it  If $s_0,\overline{s}_0 \in I$, $s_0<\overline{s}_0$ are local maximum or local minimum points of $\alpha(s)$ (resp. $\eta(s)$), then $\alpha(s_0)\leq\alpha(\overline{s}_0)$ (resp. $\eta(s_0)\geq \eta(\overline{s}_0)$). \rm In fact, since $\tau(s_0)=\tau(\overline{s}_0)=0$, $\displaystyle 2\alpha(s_0)\eta(s_0)=1$, $\displaystyle 2\alpha(\overline{s}_0)\eta(\overline{s}_0)=1$ and $g(s)=c\alpha(s)+\eta(s)$ is a  decreasing function, we have $g(s_0)\geq g(\overline{s}_0)$  and 
	$$\frac{2c\alpha^{2}(s_0)+1}{2\alpha(s_0)}\geq \frac{2c\alpha^{2}(\overline{s}_0)+1}{2\alpha(\overline{s}_0)}.$$
	  It follows from Lemma \ref{ln1} that the local maximum (resp. minimum) values of $\alpha$ are positive (resp. negative). Thus, $\alpha(s_0)\alpha(\overline{s}_0)>0$ and $\displaystyle 2c\alpha(s_0)\alpha(\overline{s}_0)\left[\alpha(s_0)-\alpha(\overline{s}_0) \right]\geq \alpha(s_0)-\alpha(\overline{s}_0).$ If $\alpha(s_0)>\alpha(\overline{s}_0)$ we would have $2c\alpha(s_0)\alpha(\overline{s}_0)\geq1$, this is a contradiction, since  $c<0$ and $\alpha(s_0)\alpha(\overline{s}_0)>0$  i.e. $\alpha(s_0)\leq\alpha(\overline{s}_0)$.  In a similar way we can prove the other inequality of the Claim.
	
	Let $\displaystyle (s_{k})_{k \in \N}\subset (\omega_{-},\overline{s})$ be such that $\tau(s_{k})=0$ for all $k \in \N$. It follows from our Claim that the local maximum values of $\alpha(s)$ and the local minimum values of $\eta(s)$ are bounded i.e. $0<\alpha(s_{2k})\leq\alpha(s_{2}) $ and $0>\eta(s_{2k+1})\geq \eta(s_{1}) $ for all $k \in \N$. Moreover, since $\alpha(s_{2k+3})\leq \alpha(s_{2k+1})<0$ then there exists $t_{2k}\in  [s_{2k+1},s_{2k}]$ such that $\displaystyle \alpha(t_{2k})=\alpha(s_1)$ for all $k \in \N$. Since $\tau(s)$ is also bounded on $[s_{2k+1},s_{2k}]$ it follows that  $\tau(t_{2k})$ is bounded for all $k \in \N$. Hence, it follows from $2\alpha(s)\eta(s)+\tau^{2}(s)=1$ that $\eta(t_{2k})$ is also bounded for all $k \in \N$.  Therefore, the sequence $(g(t_{2k}))_{k \in N}$ is monotone and bounded, and consequently the function $g(s)$ is bounded on $(\omega_{-},\overline{s})$ i.e. there exists $M \in \R^{+}$ such that $|g(s)|\leq M$ for all $s<\overline{s}$. Thus, 
	\begin{equation*}\left\{\begin{array}{l}
	\displaystyle |c\alpha(s_{2k+1})|\leq |g(s_{2k+1})| +|\eta(s_{2k+1})|\leq |\eta(s_1)|+M,\\
	\displaystyle |\eta(s_{2k})|\leq |g(s_{2k})| +|c\alpha(s_{2k})|\leq M+|c|\alpha(s_{2})	
	\end{array}\right.
	\end{equation*}
	for all $k \in \N$ i.e. the local minimum values of $\alpha(s)$ and the local maximum values of $\eta(s)$ are also bounded. Therefore, $\alpha(s)$ and $\eta(s)$ are bounded on $(\omega_{-},\overline{s})$, and it follows from $2\alpha(s)\eta(s)+\tau^{2}(s)=1$ that $\tau(s)$ is also bounded. Using similar arguments we can prove that $\alpha(s)$, $\tau(s)$ and $\eta(s)$ are bounded on $(\overline{s},\omega_{+})$.	
\end{proof}

\begin{lemma}\label{ln6}
	Let $\psi(s)=(\alpha(s),\tau(s),\eta(s))$ be a non trivial solution of (\ref{s2}) defined on the maximal interval $I=(\omega_{-}, \omega_{+})$, $a>0$, $c \in \R$ and initial condition $\psi(0)\in H\cup C\cup S$.
	\begin{itemize}
		\item[i)] If $c\geq0$, then there exist $s_1,s_2 \in I$, $s_1\leq s_2$,  such that the functions $\alpha(s),$ $\tau(s)$ and $\eta(s)$ are monotone on the intervals $(\omega_{-},s_1)$ and $(s_2,\omega_{+}).$
		\item[ii)] If $c<0$ and at least one the functions $\alpha(s)$, $\tau(s)$ and $\eta(s)$ is unbounded on $(\omega_{-}, \overline{s})$ (resp. $(\overline{s},\omega_{+})$) with $\overline{s} \in I$ fixed, then there exists $s_1 \in (\omega_{-}, \overline{s})$ (resp. $s_2 \in (\overline{s},\omega_{+})$) such that the functions $\alpha(s)$, $\tau(s)$ and $\eta(s)$ are monotone on $(\omega_{-}, s_1)$ (resp. $(s_2,\omega_{+})$).
	\end{itemize}	
\end{lemma}
\begin{proof}
	\it i) \rm 
If $c\geq 0$, it follows from Lemma \ref{ln4} that 		
		 $\tau(s)$ and $c+a\tau(s)$ have at  most two zeros. Therefore, it follows from \eqref{s2} that $\alpha(s)$ has at  most two critical points and $\eta(s)$ has at  most four critical points. We claim that $\tau(s)$ has at most three critical points. In fact, assume by contradiction that $s_1,s_2,s_3$ and $s_4$ are four consecutive critical points of $\tau(s)$ and $\tau^{\prime}(s)\neq 0$ for all $s \in (s_1,s_2)\cup (s_2,s_3)\cup (s_3,s_4)$. If $\tau(s)>0$, for $s\in J\subset I$, since $[c+\tau(s)]\tau(s)>0$, it follows from Lemma \ref{ln3} that $\tau(s)$ has no local maximum in $J$. Hence, $\tau(s)<0$ for $s \in [s_1,s_4]$ and  Lemma \ref{ln3} implies that $c+a\tau(s_{i})>0$ when $s_i$ is a local maximum point of $\tau(s)$ and $c+a\tau(s_{i})<0$ when $s_i$ is a local minimum point of $\tau(s)$, where $i \in \{1,2,3,4\}$.  However, this implies that $c+a\tau(s)$ has three zeros, which  contradicts the fact that $c+a\tau(s)$ have at  most two zeros. Hence, $\tau(s)$ has at  most three critical points. Therefore, there exist $s_1,s_2 \in I$, $s_1\leq s_2$  such that $\alpha(s),$ $\tau(s)$ and $\eta(s)$ are monotone on the intervals $(\omega_{-},s_1)$ and $(s_2,\omega_{+}).$ 
	
	\it ii) \rm If $c<0$, it follows from Lemma \ref{ln5} that, when one of the functions $\alpha(s)$, $\tau(s)$ and $\eta(s)$ is unbounded on $(\omega_{-}, \overline{s})$ (resp. $(\overline{s},\omega_{+})$) then $\tau(s)$ has at most a finite number of zeros on $(\omega_{-}, \overline{s})$ (resp. $(\overline{s},\omega_{+})$). 
Hence, $\alpha$ and  $\eta$ have a finite number of critical points.
 	
 Therefore, there exists $s_1 \in (\omega_{-}, \overline{s})$ (resp. $s_2 \in (\overline{s},\omega_{+})$) such that the functions $\alpha(s)$, $\tau(s)$ and $\eta(s)$ are monotone on $(\omega_{-}, s_1)$ (resp. $(s_2,\omega_{+})$).  
\end{proof}

\begin{lemma}\label{c41}
	Let $\psi(s)=(\alpha(s),\tau(s),\eta(s))$ be a non trivial solution of (\ref{s2}) defined on the maximal interval $I=(\omega_{-}, \omega_{+})$, $a>0$, $c \in \R$ and initial condition $\psi(0)\in H\cup C\cup S$. Let $\overline{s} \in I$.
	\item[i)]If $\alpha(s)$ is bounded on $(\omega_{-},\overline{s})$ (resp. $(\overline{s}, \omega_{+})$), then the functions $\tau(s)$ and $\eta(s)$ are bounded on $(\omega_{-},\overline{s})$ and $\omega_{-}=-\infty$ (resp. $(\overline{s}, \omega_{+})$ and $\omega_{+}=+\infty$).
	\item[ii)]If $\eta(s)$ is bounded on $(\omega_{-},\overline{s})$ (resp. $(\overline{s}, \omega_{+})$), then $\tau(s)$ is bounded on $(\omega_{-},\overline{s})$ and $\omega_{-}=-\infty$ (resp. $(\overline{s}, \omega_{+})$ and $\omega_{+}=+\infty$).
\end{lemma}

\begin{proof}
	Note that, if $\alpha(s)\neq 0$, $\eta(s)\neq 0$ and  $[c+a\tau(s)]\tau(s)\neq 0$, since $\psi(s)\in H\cup C\cup S$, we have   $2\alpha(s)\eta(s)+\tau^{2}(s)=\gamma$, where $\gamma \in \{-1,0,1\}$ and  
	it follows from \eqref{s2} that 
	\begin{equation}\label{eb1}
	\frac{\displaystyle{\frac{d}{ds}}\left(\tau^{2}(s)\right)}{\displaystyle{\frac{d}{ds}}\left(2\eta(s)\right)}=-\alpha(s)+\frac{\gamma}{2\alpha(s)[c+a\tau(s)]}-\frac{\tau(s)}{2\alpha(s)\left[\displaystyle{\frac{c}{\tau(s)}}+a\right]}.
	\end{equation}
	
	\it i) \rm We will only consider the interval $(\omega_{-},\overline{s})$, since the proof for the interval $(\overline{s},\omega_{+})$ follows from similar arguments.  Suppose that $\alpha(s)$ is bounded on $(\omega_{-},\overline{s})$. Assume by contradiction that $\tau(s)$ is unbounded on $(\omega_{-},\overline{s}).$ It follows from Lemma \ref{ln6} that  there exists $s_1\in (\omega_{-},\overline{s})$ such that $\alpha(s)$, $\tau(s)$ and $\eta(s)$ are monotone on $(\omega_{-},s_1)$. Thus,    $$\displaystyle \lim_{s \to \omega_{-}}\tau(s)=\pm \infty, \hspace{0.3 cm} \displaystyle \lim_{s \to \omega_{-}}\tau(s)[c+a\tau(s)]=+ \infty \hspace{0.3 cm} \text{and}\hspace{0.3 cm} \eta^{\prime}(s)=-[c+a\tau(s)]\tau(s)<0$$ 
	for all $s<s_1$. Hence, it follows from $2\alpha(s)\eta(s)+\tau^{2}(s)=\gamma$ that $\eta(s)$ is also unbounded, $ \lim_{s \to \omega_{-}}\eta(s)=+ \infty$ and $ \lim_{s \to \omega_{-}}\eta(s)\alpha(s)=- \infty$. Moreover, $s_1$ can be chosen such that $\alpha(s)<0$ for all $s<s_1.$ 
	
	Suppose that $ \lim_{s \to \omega_{-}}\tau(s)=- \infty$, then there exists  $s_1$ such that $\tau^{\prime}(s)>0$, $\gamma-\tau^{2}(s)<0$  and $c+a\tau(s)<0$ for all $s<s_1$, where $\gamma \in \{-1,0,1\}.$ Thus, $\alpha(s)\tau^{\prime}(s)<0$ and \eqref{novo1} implies that
	$$
	\alpha(s)\tau^{\prime}(s)=[c+a\tau(s)]\alpha^{2}(s)-\frac{\gamma}{2}+\frac{\tau^{2}(s)}{2}<0 \,\,\,\,\,\, \text{i.e.}\,\,\,\,\,\, \alpha^{2}(s)>\frac{\gamma-\tau^{2}(s)}{c+a\tau(s)}>0
	$$
	for all $s<\overline{s}_1$, this is a contradiction, we assumed that $ \lim_{s \to \omega_{-}}\tau(s)=- \infty$ and $\alpha(s)$ is bounded. 
	
	Suppose that $ \lim_{s \to \omega_{-}}\tau(s)=+ \infty$. Then there exists $s_1\in (\omega_{-},\overline{s})$ such that $\alpha(s)$ is increasing and negative on $(\omega_{-},s_1)$. Since that $\alpha(s)$ is bounded, we have $ \lim_{s \to \omega_{-}}\alpha(s)=L$, where $L \in \R^{-}$ and $ \lim_{s \to \omega_{-}}\alpha(s)[c+a\tau(s)]=-\infty$. Thus, it follows from \eqref{eb1} and L'Hospital rule that 
	$$\displaystyle \lim_{s \to \omega_{-}} \frac{\displaystyle{\frac{d}{ds}}\left(\tau^{2}(s)\right)}{\displaystyle{\frac{d}{ds}}\left(2\eta(s)\right)}=+\infty \,\,\,\, \text{and}\,\,\,\,\displaystyle \lim_{s \to \omega_{-}} \frac{\tau^{2}(s)}{2\eta(s)}=\displaystyle \lim_{s \to \omega_{-}} \frac{\displaystyle{\frac{d}{ds}}\left(\tau^{2}(s)\right)}{\displaystyle{\frac{d}{ds}}\left(2\eta(s)\right)}=+\infty.$$
	Since  $2\alpha(s)\eta(s)+\tau^{2}(s)=\gamma$, then 
	$\lim_{s \to \omega_{-}} \alpha(s)=\lim_{s\to\omega_{-}}[\gamma/(2\eta(s))-\tau^{2}/(2\eta(s))]=-\infty$. This  contradicts that $\alpha(s)$ is bounded on $(\omega_{-},\overline{s})$.  Therefore, $\tau(s)$ is bounded on $(\omega_{-},\overline{s}).$	
	
	Let $M>0$ be such that $|\tau(s)|<M$, $s\in (\omega_{-},\overline{s})$. It follows from (\ref{s2})  that
	\begin{equation*}
	|\eta(\overline{s})-\eta(s)|=\left| -\int^{\overline{s}}_{s}[c+\tau(u)]\tau(u)du \right|<\left(|c|M+aM^{2}\right)(\overline{s}-s)
	\end{equation*}
	for each $s<\overline{s}$. Since $I$ is a maximal interval we conclude that the solution $\psi(s)$ of (\ref{s2}) leaves any compact subset of $(\omega_{-},\overline{s}) \times \R^{3}$, then $\omega_{-}=-\infty$.  
	
	We will now proof that $\eta(s)$ is bounded. Assume by contradiction that $\eta(s)$ is unbounded. It follows from Lemma \ref{ln6} that there exists  $s_1$ such that the functions $\alpha(s)$, $\tau(s)$ are $\eta(s)$ are monotone on $(-\infty,s_1)$ and $ \lim_{s \to -\infty}\eta(s)=\pm \infty$. Moreover, the limits $\lim_{s \to -\infty}\alpha(s)$ and $ \lim_{s \to -\infty}\tau(s)$ exist. Hence, the improper integral  $$\int^{\overline{s}}_{-\infty}\tau^{\prime}(u)du=\tau(\overline{s})-\lim_{s \to -\infty}\tau(s)$$ is convergent, which  is a contradiction, since $ \lim_{s \to -\infty} [c+a\tau(s)]\alpha(s)-\eta(s)=\mp \infty$. Therefore, $\eta(s)$ is also bounded on $(-\infty,\overline{s})$.
	
	\it ii) \rm We will give the proof only for the interval $(\omega_{-},\overline{s})$, since the proof for the interval $(\overline{s},\omega_{+})$ follows by similar arguments. Let $\eta(s)$ be bounded on $(\omega_{-},\overline{s})$. Assume by contradiction that $\tau(s)$ is unbounded on $(\omega_{-},\overline{s})$. It follows from Lemma \ref{ln6} that there exists $s_1\in (\omega_{-},\overline{s})$ such that $\alpha(s)$, $\tau(s)$ and $\eta(s)$ are monotone on $(\omega_{-},s_1)$ and $ \lim_{s \to \omega_{-}}\tau(s)=\pm \infty$. Since $2\alpha(s)\eta(s)+\tau^{2}(s)=\gamma$,   we have that $\alpha(s)$ is also unbounded, $ \lim_{s \to \omega_{-}}\alpha(s)=\pm \infty$ and $ \lim_{s \to \omega_{-}}\tau^{\prime}(s)= \lim_{s \to \omega_{-}}[c+a\tau(s)]\alpha(s)-\eta(s)=\pm \infty.$ Using  L'Hospital rule we obtain
	$$\displaystyle \lim_{s \to \omega_{-}} \frac{\tau^{2}(s)}{2\alpha(s)}=\lim_{s \to \omega_{-}} \frac{2\tau(s)\tau^{\prime}(s)}{2\tau(s)}=\lim_{s \to \omega_{-}}\tau^{\prime}(s)=\pm \infty.$$
	Hence, 
	$ \lim_{s \to \omega_{-}} \eta(s)=\lim_{s \to \omega_{-}} [\gamma/{(2\alpha(s))}-\tau^{2}{(2\alpha(s))}]=\mp\infty$, which 
contradicts the hypothesis thet $\eta(s)$ is bounded on $(\omega_{-},\overline{s})$.  Therefore, $\tau(s)$ is bounded on $(\omega_{-},\overline{s}).$ Let $M>0$ be such that $|\tau(s)|<M$ for all $s\in (\omega_{-},\overline{s})$. From (\ref{s2}) we obtain
	\begin{equation*}
	|\alpha(\overline{s})-\alpha(s)|=\left| \int^{\overline{s}}_{s}\tau(u)du \right|<M(\overline{s}-s)
	\end{equation*}
	for each $s<\overline{s}$. Since $I$ is a maximal interval, the solution $\psi(s)$ of (\ref{s2}) leaves any compact subset of $(\omega_{-},\overline{s}) \times \R^{3}$. Hence $\omega_{-}=-\infty$. 			\end{proof}

\begin{lemma} \label{g2}
	Let $\psi(s)=(\alpha(s),\tau(s),\eta(s))$ be a non trivial solution of (\ref{s2}) defined on the maximal interval $I=(\omega_{-}, \omega_{+})$, $a>0$, $c \in \R$ and initial condition $\psi(0)\in H\cup C\cup S$. If $\alpha(s)$ is bounded on the interval $(\omega_{-},\overline{s})$ (resp. $(\overline{s},\omega_{+})$), $\overline{s} \in I$ fixed, then $\omega_{-}=-\infty$, $ \lim_{s \to -\infty}\tau(s)=0$ and $ \lim_{s \to -\infty}\tau^{\prime}(s)=0$ (resp. $\omega_{+}=+\infty$, $ \lim_{s \to +\infty}\tau(s)=0$ and $ \lim_{s \to +\infty}\tau^{\prime}(s)=0.$)
\end{lemma}
\begin{proof} We will only consider the interval $(\omega_{-},\overline{s})$, since the proof for the interval $(\overline{s},\omega_{+})$ follows with  similar arguments. We define the function $g(s)=c\alpha(s)+\eta(s)$, $s \in I$.  It follows from \eqref{s2} that  $g^{\prime}(s)=-a\tau^{2}(s)$ and $g(s)$ is decreasing for all $s \in I$. 
	
 Suppose that $\alpha(s)$ is bounded on $(\omega_{-},\overline{s})$, $\overline{s} \in I$ fixed, then it follows from item \it i) \rm of  Lemma \ref{c41} that $\eta(s)$ and $\tau(s)$ are bounded on $(\omega_{-},\overline{s})$ and $\omega_{-}=-\infty$. Thus, the function $g(s)$ is decreasing and bounded on $(-\infty,\overline{s})$ and $ \lim_{s \to -\infty}g(s)$ exists. Note that $g^{\prime}(s)$ is uniformly continuous on $(-\infty,\overline{s})$, since  $g^{\prime \prime}(s)=2a\tau(s)[c+a\tau(s)]\alpha(s)-\eta(s)$ is bounded on $(-\infty,\overline{s})$. Hence, it follows from Barbalat's Lemma that $ \lim_{s \to -\infty}g^{\prime}(s)=0$. Since $\displaystyle g^{\prime}(s)=-a\tau^{2}(s)$,  we conclude that $ \lim_{s \to -\infty}\tau(s)=0$. Moreover, it follows from  \eqref{s2} that 
 $\tau^{\prime \prime}(s)$ is also bounded on $(-\infty,\overline{s})$ i.e. $\tau^{\prime}(s)$ is uniformly continuous on $(-\infty, \overline{s})$. Since $ \lim_{s \to -\infty}\tau(s)=0$ then it follows from Barbalat's Lemma that $ \lim_{s \to -\infty}\tau^{\prime}(s)=0.$
\end{proof}

\begin{lemma} \label{g3}
	Let $\psi(s)=(\alpha(s),\tau(s),\eta(s))$ be a non trivial solution of (\ref{s2}) defined on the maximal interval $I=(\omega_{-}, \omega_{+})$, $a>0$, $c \in \R$ and initial condition $\psi(0)\in H\cup C\cup S$. Let $\overline{s} \in I$ be fixed. If $\alpha(s)$ is unbounded on $(\omega_{-},\overline{s})$ , then either $\lim_{s \to \omega_{-}}|\eta(s)|=\lim_{s \to \omega_{-}}|c+a\tau(s)|=+\infty$ or $ \lim_{s \to -\infty}\eta(s)=\lim_{s \to -\infty}c+a\tau(s)=0$. If $\alpha(s)$ is unbounded on $(\overline{s},\omega_{+})$, then 	
	 either $\lim_{s \to \omega_{+}}|\eta(s)|=\lim_{s \to \omega_{+}}|c+a\tau(s)|=+\infty$ or $\lim_{s \to +\infty}\eta(s)= \lim_{s \to +\infty}c+a\tau(s)=0$.
\end{lemma}
\begin{proof} We will only consider the interval $(\omega_{-},\overline{s})$, since the proof for the interval $(\overline{s},\omega_{+})$ follows with similar arguments.
	Since $\alpha(s)$ is unbounded on $(\omega_{-},\overline{s})$ it follows from Lemma \ref{ln6} that there exists $s_1 \in (\omega_{-},\overline{s})$ such that $\alpha(s)$, $\tau(s)$ and $\eta(s)$ are monotone on $(\omega_{-},s_1)$. Thus, the monotone function $\eta(s)$ is either bounded or unbounded on $(\omega_{-},s_1)$. 
	
	If $\eta(s)$ is bounded, it follows item \it ii) \rm of  Lemma \ref{c41} that $\tau(s)$ is bounded on $(\omega_{-},s_1)$, $\omega_{-}=-\infty$ and the limits $ \lim_{s \to -\infty}\tau(s)$ and $ \lim_{s \to -\infty}\eta(s)$ exist. Hence, $ \lim_{s \to -\infty}\eta^{\prime}(s)=-\lim_{s \to -\infty}\tau(s)[c+a\tau(s)]$ also exist and therefore $ \lim_{s \to -\infty}\tau(s)[c+a\tau(s)]=0.$ We will prove that $ \lim_{s \to -\infty}\tau(s) \neq 0$ when $c\neq 0$. Assume by contradiction that $ \lim_{s \to -\infty}\tau(s)=0$. Hence, $ \lim_{s \to -\infty}\tau^{\prime}(s)=\lim_{s \to -\infty}[c+\tau(s)]\alpha(s)-\eta(s)=\pm \infty$ and the improper integral 
	$$\int_{-\infty}^{s_1}\tau^{\prime}(u)du=\tau(s_1)-\lim_{s \to -\infty}\tau(s)=\tau(s_1),$$
which is a contradiction. Therefore, $ \lim_{s \to -\infty}[c+a\tau(s)]=0$. 
	Note that  $2\alpha(s)\eta(s)+\tau^{2}(s)=\gamma,$ where $\gamma \in \{-1,0,1\}$. By hypothesis   $\alpha(s)$ is unbounded,  hence  $ \lim_{s \to -\infty}\eta(s)=0$. 
	
	If $\eta(s)$ is unbounded on $(\omega_{-},s_1)$ then $ 2\alpha(s)\eta(s)+\tau^{2}(s)=\gamma$ implies that $\tau(s)$ is unbounded on $(\omega_{-},s_1)$ and thus $\lim_{s \to \omega_{-}}|\eta(s)|=\lim_{s \to \omega_{-}}|c+a\tau(s)|=+\infty$.
	\end{proof}

\begin{lemma}\label{c40l3}
	Let $\psi(s)=(\alpha(s),\tau(s),\eta(s))$ be a non trivial solution of (\ref{s2}) defined on the maximal interval $I=(\omega_{-}, \omega_{+})$, $a>0$, $c\geq0$ and initial condition $\psi(0)\in H$. 
	Then there exists a global maximum point $s_0$ of $\alpha(s)$.  
\end{lemma}

\begin{proof}
	Assume by contradiction that $\alpha(s)$ has no critical  points i.e. $\tau(s)>0$ or $\tau(s)<0$ for all $s \in I.$		
	If  $\tau(s)>0$, then $\alpha(s)$ is negative and strictly increasing on $I$ and $\eta(s)$ is positive and strictly decreasing on $I$.  Thus, the limits $ \lim_{s \to \omega_{+}}\alpha(s)$, $ \lim_{s \to \omega_{+}}\eta(s)$ exist and consequently  $ \lim_{s \to \omega_{+}}\tau(s)$ exists, since $-2\alpha(s)\eta(s)=\tau^{2}(s)+1$. Hence, there exists $p \in H$ such that $ \lim_{s \to \omega_{+}}\psi(s)=p$ and thus, $\omega_{+}=+\infty$ and $\psi(s)=p$ is a singular solution of \eqref{s2}. This contradicts Lemma \ref{tri} which shows  that the system \eqref{s2} has no singular solutions on the set $H$, when  $c\geq 0$. If  $\tau(s)<0$, we  obtain a similar contradiction. Therefore, there exists $s_0 \in I$ such that $s_0$ is a critical point of $\alpha(s)$ and it follows from item \it i) \rm of  Lemma \ref{ln1} that $s_0$ is the global maximum point.
\end{proof}

 Our next lemma shows that for non trivial solutions  $\psi$  of \eqref{s2},  when $c\geq 0$ (resp. $c\leq 0$) and $\psi\in H$ 
(resp. $\psi\in S$), then $\alpha$ is unbounded on both ends of the maximal interval $I$. 

\begin{lemma}\label{c40l4extra}
	Let $\psi(s)=(\alpha(s),\tau(s),\eta(s))$ be a non trivial solution of (\ref{s2}) defined on the maximal interval $I=(\omega_{-}, \omega_{+})$, with  $a>0$. 
	
	i) If $c\geq0$ and $\psi(0)\in H$,  then $ \lim_{s \to \omega_{-}}\alpha(s)=\lim_{s \to \omega_{+}}\alpha(s)=-\infty$.

ii) If $c\leq0$ and $\psi(0)\in S$, then $ \lim_{s \to \omega_{-}}|\alpha(s)|=\lim_{s \to \omega_{+}}|\alpha(s)|=+\infty$.
\end{lemma}

\begin{proof}
\it i)  \rm	Lemma \ref{c40l3} implies that $\alpha(s)$ has a unique critical point $s_0\in I$. It follows from Lemma \ref{ln3} that $\tau(s)$ has at  most one critical point $s_1$ and $\tau(s_1)<0$. Thus, the functions $\alpha(s)$, $\tau(s)$ and $\eta(s)$ are monotone on $(\omega_{-},s_0)$ and $(s_1,\omega_{+})$. Assume by contradiction that $\alpha(s)$ is bounded on $(\omega_{-},s_0)$ (resp. $(s_1,\omega_{+})$). It follows from item \it i) \rm of  Lemma \ref{c41} that $\tau(s)$ and $\eta(s)$ are also bounded on $(\omega_{-},s_0)$ and $\omega_{-}=-\infty$ (resp. $(s_1,\omega_{+})$ and $\omega_{+}=+\infty$). 
	Hence, there exist $p \in H$ (resp. $\overline{p}\in H$) such that $ \lim_{s \to -\infty}\psi(s)=p$, (resp. $\lim_{s \to +\infty}\psi(s)=\overline{p}$) and $\psi(s)=p$ (resp. $\psi(s)=\overline{p}$) is a singular solution of \eqref{s2} in $H$. This contradicts Lemma \ref{tri} that asserts that  \eqref{s2} has no singular solutions in the set $H$, when $c\geq0$. Since $\alpha(s)<0$ we conclude that $ \lim_{s \to \omega_{-}}\alpha(s)=\lim_{s \to \omega_{+}}\alpha(s)=-\infty.$

\it ii) \rm 	Let us consider $\overline{s} \in I$ fixed. We will only consider the interval $(\omega_{-},\overline{s})$, since   the proof for the interval $(\overline{s},\omega_{+})$ follows from similar arguments. Assume by contradiction that $\alpha(s)$ is bounded on $(\omega_{-},s_0)$. It follows from item \it i) \rm of  Lemma \ref{c41} that $\tau(s)$ and $\eta(s)$ are also bounded on $(\omega_{-},s_0)$ and $\omega_{-}=-\infty$. Thus, Lemma \ref{g2} implies that $\lim_{s \to -\infty}\tau(s)=0$ and $\lim_{s \to -\infty}\tau^{\prime}(s)=0$. From \eqref{novo1} and \eqref{novo2} we have  that $\lim_{s \to -\infty}\alpha(s)$ and $\lim_{s \to -\infty}\eta(s)$ exist. Hence, there exists $p \in S$ such that $ \lim_{s \to -\infty}\psi(s)=p$ and $\psi(s)=p$, $s \in \R$  is a singular solution of \eqref{s2} in $S$, which contradicts  Lemma \ref{tri} that asserts that \eqref{s2} has no singular solutions in $S$, when $c\leq0$. Therefore, $\alpha(s)$ is unbounded on $(\omega_{-},\overline{s})$ and it follows from item \it ii) \rm of  Lemma \ref{ln6} that $ \lim_{s \to \omega_{-}}|\alpha(s)|=+\infty.$
\end{proof}

In order to obtain additional properties on the solutions of  (\ref{s2}),  we will separate the study in three cases: $c=0$, $c<0$ and $c>0$. 

When $c=0$ the solutions of  \eqref{s2} provide the   soliton solutions to the CF in $Q^{2}_{+}$. In this case,  (\ref{s2}) has no singular solutions in the set $H\cup S$. We also recall that, if $\psi(s) \in H$, $s \in I$, then $\alpha(s)<0$ and $\eta(s)>0$, and if $\psi(s) \in C$, $s \in I$, then $\alpha(s)\leq0$ and $\eta(s)\geq0$. Note that, when $c=0$, $\psi(s)=(\alpha_0,0,0)$, $s \in I$, $\alpha_{0}\neq0$ are singular solutions of (\ref{s2}) in $C$. These solutions correspond to the  parables on the light cone i.e. they are trivial solutions to the CF in $Q^{2}_{+}$. In the next lemma we will study the behaviour of $\tau(s)$, when $c=0$ and $\psi(s) \in H\cup C$  is a non trivial solution of \eqref{s2}.

\begin{lemma}\label{c40l5}
	Let $\psi(s)=(\alpha(s),\tau(s),\eta(s))$ be a non trivial solution of (\ref{s2}) defined on the maximal interval $I=(\omega_{-}, \omega_{+})$, $a>0$, $c=0$ and initial condition $\psi(0)\in H\cup C$. Then $\omega_{+}=+\infty$,  $ \lim_{s \to +\infty}\tau(s)=0$ and $ \lim_{s \to \omega_{-}}\tau(s)=+\infty$.
\end{lemma}

\begin{proof}
	It follows from item \it i) \rm of  Lemma \ref{ln6} that there exist $s_1,s_2 \in I$ such that $\alpha(s)$, $\tau(s)$ and $\eta(s)$ are monotone on $(\omega_{-},s_1)$ and $(s_2,\omega_{+})$. Since $\eta(s)$ is a non  negative function when $\psi(0) \in H\cup C$, the third equation of \eqref{s2} and  $c=0$ imply that $\eta(s)$ is decreasing and bounded on $(s_2, \omega_{+})$. Thus, it follows from item \it ii) \rm of  Lemma \ref{c41} that $\tau(s)$ is bounded on $(s_2, \omega_{+})$ and $\omega_{+}=+\infty$. Hence, the limits $ \lim_{s \to +\infty}\eta(s),$ $ \lim_{s \to +\infty}\tau(s)$ exist. Therefore, $ \lim_{s \to +\infty}\eta^{\prime}(s)=-\lim_{s \to +\infty}a\tau^{2}$ exists and consequently $ \lim_{s \to +\infty}\tau(s)=0$.
	
	We will now prove that $\tau(s)$ is unbounded on $(\omega_{-}, s_1)$.  Lemma \ref{c40l4extra} asserts that $\alpha$ is unbounded on $(\omega_{-}, s_1)$ when $\psi(0) \in H$ and $\lim_{s \to \omega_{-}}\alpha(s)=-\infty$. We claim that $\alpha(s)$ is also unbounded on $(\omega_{-},s_1)$ when $\psi(0) \in C$. In fact, assume by contradiction that $\alpha(s)$ is bounded on $(\omega_{-},s_1).$ Then, it follows from item \it i) \rm of  Lemma \ref{c41} that $\tau(s)$ and $\eta(s)$ are also bounded on $(\omega_{-},s_1)$ and $\omega_{-}=-\infty$. Moreover, Lemma \ref{g2} implies that $ \lim_{s \to -\infty}\tau(s)=0$ and $ \lim_{s \to -\infty}\tau^{\prime}(s)=0$. Thus, $ \lim_{s \to -\infty}\eta(s)=\lim_{s \to -\infty}\tau^{\prime}(s)-\alpha(s)\tau(s)=0.$ Since $\eta(s)$ is non negative and decreasing we conclude that $\eta(s)$ vanishes  on $(\omega_{-},s_1)$, which contradicts  the assumption that we are considering only non trivial solutions. Hence, $\alpha(s)$ is unbounded when $\psi(0) \in C$ and $\lim_{s \to \omega_{-}}\alpha(s)=-\infty$ i.e. $\alpha(s)$ is increasing on $(\omega_{-},s_1)$. Therefore, if $\psi(0) \in H \cup C$ and $\psi(s)$ is a non trivial solution then $ \lim_{s \to \omega_{-}}\tau(s)=+\infty$.	
\end{proof}

We will now study the solutions of \eqref{s2} with initial condition in $S$. In this case,  (\ref{s2}) has no singular solutions. Moreover, the functions $\alpha(s)$ and $\eta(s)$ may change sign.

\begin{lemma}\label{c40l8}
	Let $\psi(s)=(\alpha(s),\tau(s),\eta(s))$ be a non trivial solution of (\ref{s2}) defined on the maximal interval $I=(\omega_{-}, \omega_{+})$, $a>0$, $c=0$ and initial condition $\psi(0)\in S$. Then, 
	
	i) either $\lim_{s \to \omega_{-}}\eta(s)=\lim_{s \to \omega_{-}}\tau(s)=+\infty$ or $\lim_{s \to -\infty}\eta(s)=\lim_{s \to -\infty}\tau(s)=0$ and 
	
	ii) either $\lim_{s \to \omega_{+}}-\eta(s)= \lim_{s \to \omega_{+}}\tau(s)=+\infty$ or $ \lim_{s \to +\infty}\eta(s)=\lim_{s \to +\infty}\tau(s)=0$). 
\end{lemma}

\begin{proof}
	Lemma \ref{c40l4extra} asserts that $\alpha(s)$ is unbounded on $(\omega_{-},\overline{s})$ and $(\overline{s},\omega_{+})$, for  $\overline{s} \in I$. Thus, it follows from Lemma \ref{g3} that either $\lim_{s \to \omega_{-}}|\eta(s)|=\lim_{s \to \omega_{-}}|\tau(s)|=+\infty$ or $\lim_{s \to -\infty}\eta(s)=\lim_{s \to -\infty}\tau(s)=0$. Moreover,  either $\lim_{s \to \omega_{+}}|\eta(s)|= \lim_{s \to \omega_{+}}|\tau(s)|=+\infty$ or $ \lim_{s \to +\infty}\eta(s)=\lim_{s \to +\infty}\tau(s)=0$). The third equation of \eqref{s2} implies that $\eta(s)$ is a decreasing function on $I$. Thus, when $\tau(s)$ is unbounded on $(\omega_{-},\overline{s})$  (resp. $(\overline{s},\omega_{+})$) then $\eta(s)$ is unbounded  on $(\omega_{-},\overline{s})$ (resp. $(\overline{s},\omega_{+})$) and $\lim_{s \to \omega_{-}}\eta(s)=+\infty$ (resp. $\lim_{s \to \omega_{+}}\eta(s)=-\infty$). Thus, $2\alpha(s)\eta(s)+\tau^{2}(s)=1$ implies that $\lim_{s \to \omega_{-}}\alpha(s)=-\infty$ (resp. $\lim_{s \to \omega_{+}}\alpha(s)=+\infty$) and therefore, $\lim_{s \to \omega_{-}}\tau(s)=+\infty$ (resp. $\lim_{s \to \omega_{+}}\tau(s)=+\infty$).
\end{proof}
 
Let $X(s)$ be a soliton solution  to the CF on $ Q^{2}_{+}$ corresponding to a solution of \eqref{s2}.  Since its curvature is given by $k(s)=a\tau(s)$, as an immediate consequence of Lemmas  \ref{c40l5} and \ref{c40l8},  we get 

\begin{corollary}\label{coroc3_4}
 Assume that  the curve $X: I \rightarrow Q^{2}_{+}$, $s \in I$ is a soliton solution to the CF, with curvature $k(s)$, that corresponds to a non trivial solution  $\psi(s)=(\alpha(s),\tau(s),\eta(s))$  of (\ref{s2}) defined on the maximal interval $I=(\omega_{-}, \omega_{+})$, $a>0$, $c=0$, with initial condition $\psi(0)\in H\cup C\cup S$.  Then 

i) If $\psi(0)\in H\cup C$ (resp. $\psi(0)\in C$) then  $\omega_+=+\infty$.    Moreover 	$ \lim_{s \to \omega_{-}}k(s)=+\infty$ and $ \lim_{s \to +\infty}k(s)=0.$ 

ii) If $\psi(0)\in S$, then at each end of the curve    $k(s)$ tends to $+\infty$ or converges to zero.	
\end{corollary}

We will now study the solutions of \eqref{s2} with $\psi(0) \in H\cup C \cup S$ and $c<0$. Note that $\psi(s)=(\alpha_0,0,c\alpha_0) \in H$, $s \in I$ is a singular solution and \eqref{s2} has no singular solutions in $C$. However $\psi(s)=(0,0,0)$ is a singular solution at the boundary of $C$. In the next lemmas we will prove that $(\alpha_0,0,c\alpha_0)$ and $(0,0,0)$ are global attractors for the solutions in $H$ and $C$ respectively.

\begin{lemma}\label{c4-1l1}
	Consider $a>0$ and $c<0$. Let  $\Phi:H\rightarrow TH \subset \R^{3}$ be the vector field given by 
	$\Phi(\alpha,\tau,\eta)=\left(\tau,c\alpha+a\tau\alpha-\eta,-c\tau-a\tau^{2}\right)$. 
	Then  the singular point $p$ of $\Phi$ and the eigenvalues $\lambda$ of $d\Phi_{p}$ are given by
	\begin{equation}\label{pla} 
p=\left(-\frac{1}{\sqrt{-2c}},0,-\frac{c}{\sqrt{-2c}}\right), \qquad 	\qquad\lambda=\frac{-a\pm \sqrt{a^{2}-16c^{2}}}{2\sqrt{-2c}}.
	\end{equation}
	The eigenvalues are real numbers if $4|c|\leq a$ and they are complex numbers if $4|c|> a$.
\end{lemma}
\begin{proof}
	The singular point of $\Phi$ satisfies $\tau=0$, $c\alpha=\eta$. Since $(\alpha,\tau,\eta) \in H$ then $-1=2\alpha\eta+\tau^{2}=2\alpha\eta$. Hence, $2c\alpha^{2}=-1$, $  \alpha=- {1}/{\sqrt{-2c}}$. The singular point $ p$ of $\Phi$ in $H$ and  the eigenvalues  $\lambda$ of  $d\Phi_{ p}$  are  given by \eqref{pla} and the   eigenvectors are $ w=(w_{1},w_{2},-cw_{1})\in T_{ p}H,  w_{1}, w_{2} \in \R $. 
\end{proof}

\begin{lemma}\label{c4-1l6}
	Let $\psi(s)=(\alpha(s),\tau(s),\eta(s))$ be a non trivial solution of (\ref{s2}) defined on the maximal interval $I=(\omega_{-}, \omega_{+})$, $a>0$, $c<0$ and initial condition $\psi(0)\in H\cup C$. Then $\omega_+=+\infty$. If $\psi(0)\in H$ then  $\lim_{s \to +\infty} \psi(s)=p$, where $p$ is given by \eqref{pla}. 
	If $\psi(0)\in C$ then $ \lim_{s \to +\infty}\psi(s)=(0,0,0)$. 
	\end{lemma}
\begin{proof}
	Define  $ g(s)=c\alpha(s)+\eta(s)$, for  $s \in I$. Since $\alpha(s)\leq0$, $\eta(s)\geq0$ and $c<0$, then  $g(s)$ is a positive function and $g^{\prime}(s)=-a\tau^{2}(s)<0$, i.e.,  $g(s)$ is a decreasing function and $0 <g(s)\leq g(\overline{s})=M$, for all $s>\overline{s}\in I$ and $M \in \R$. Therefore, $0\leq c\alpha(s),\, \eta(s)<c\alpha(s)+\eta(s) < M$ , $s>\overline{s}$.  Hence, $\alpha(s)$ and $\eta(s)$ are bounded on $(\overline{s},\omega_{+})$, and  Lemma \ref{g2} implies that $\omega_{+}=+\infty$, $ \lim_{s \to +\infty}\tau(s)=0$ and $ \lim_{s \to +\infty}\tau^{\prime}(s)=0$. It follows from \eqref{novo1} and \eqref{novo2}  that  if $\psi(0) \in H$ then $\lim_{s \to +\infty} 2c\alpha^{2}(s)=-1$ and $\lim_{s \to +\infty} 2\eta^{2}(s)=-c,$ and if $\psi(0) \in C$ then $\lim_{s \to +\infty} \alpha^{2}(s)=
	\lim_{s \to +\infty} \eta^{2}(s)=0$.
\end{proof}

\begin{lemma} \label{c4-1l8l14}
	Let $\psi(s)=(\alpha(s),\tau(s),\eta(s))$ be a non trivial solution of (\ref{s2}) defined on the maximal interval $I=(\omega_{-}, \omega_{+})$, $a>0$, $c<0$ and initial condition $\psi(0)\in H\cup C\cup S$. 

i) If $\psi(0)\in H\cup C$, then	 $ \lim_{s \to \omega_{-}} \alpha(s)=-\infty$ and \\ either $\lim_{s \to \omega_{-}}\eta(s)=\lim_{s \to \omega_{-}}c+a\tau(s)=+\infty$ or $\lim_{s \to -\infty}\eta(s)=\lim_{s \to -\infty}c+a\tau(s)=0$.

ii) If $\psi(0)\in S$, then on the boundary of the interval $I$ we have: \\ 
Either $ \lim_{s \to \omega_{-}}|\eta(s)|=\lim_{s \to \omega_{-}}|c+a\tau(s)|=+\infty$ or $\lim_{s \to -\infty}\eta(s)=\lim_{s \to -\infty} c+a\tau(s)=0$. \\  
Either $  \lim_{s \to \omega_{+}}|\eta(s)|=\lim_{s \to \omega_{+}}|c+a\tau(s)|=+\infty$  or $\lim_{s \to +\infty}\eta(s)=\lim_{s \to +\infty}c+a\tau(s)=0$.
\end{lemma}
\begin{proof}
	\it i) \rm Assume by contradiction that $\alpha(s)$ is bounded on $(\omega_{-}, \overline{s})$, $\overline{s} \in I$ fixed. It follows from Lemma \ref{g2} that $\tau(s)$ and $\eta(s)$ are bounded on $(\omega_{-}, \overline{s})$, $\omega_{-}=-\infty$, $ \lim_{s \to -\infty}\tau(s)=0$ and $ \lim_{s \to -\infty}\tau^{\prime}(s)=0$. 
	Thus, using Lemma \ref{novo} we have: if $\psi(0) \in H$ then $\lim_{s \to +\infty} 2c\alpha^{2}(s)=-1$ and $\lim_{s \to +\infty} 2\eta^{2}(s)=-c,$ and if $\psi(0) \in C$ then $\lim_{s \to +\infty} \alpha^{2}(s)=\lim_{s \to +\infty} \eta^{2}(s)=0.$ Hence, $\left(-{1}/{\sqrt{-2c}},0,-{c}/{\sqrt{-2c}} \right)$ and $(0,0,0)$ are non attractor singular solutions of \eqref{s2} which contradicts  Lemmas \ref{c4-1l1} and \ref{c4-1l6}. Therefore, $\alpha(s)$ is unbounded on $(\omega_{-}, \overline{s})$. Thus, it follows from item \it ii) \rm of  Lemma \ref{ln3} that there exists $s _1 \in (\omega_{-},\overline{s})$ such that the functions $\alpha(s)$ and $\tau(s)$ are monotone on $(\omega_{-},s_1)$. Hence, $ \lim_{s \to \omega_{-}} \alpha(s)=-\infty$ and there exists $\hat{s} \in (\omega_{-},s_1)$ such that $\tau(s)>0$ for all $s<\hat{s}$. Therefore, Lemma \ref{g3} implies  that either $\lim_{s \to \omega_{-}}c+a\tau(s)=+\infty$ or $\lim_{s \to -\infty}c+a\tau(s)=\lim_{s \to -\infty}\eta(s)=0$. Moreover, if $\lim_{s \to \omega_{-}}c+a\tau(s)=+\infty$ then we can choose $s_1$ such that $\eta^{\prime}(s)=-[c+a\tau(s)]\tau(s)<0$ for all $s<s_1$ and thus, $\lim_{s \to \omega_{-}}\eta(s)=+\infty$.

\it ii) \rm When $c<0$ and $\psi\in S$, then  $\alpha(s)$ and $\eta(s)$ may change sign. Moreover, it follows from Lemma \ref{tri} that there are no singular solutions in $S$.  Lemma \ref{c40l4extra} implies that $\lim_{s \to \omega_{-}} |\alpha(s)|=+\infty$ and $ \lim_{s \to \omega_{+}} |\alpha(s)|=+\infty$. Thus, the result follows from Lemma \ref{g3}.
\end{proof}

As an immediate consequence of Lemmas \ref{c4-1l6} and  \ref{c4-1l8l14}, we get the following. 

\begin{corollary}\label{coroc5_6_9}
 Assume that a curve $X: I \rightarrow Q^{2}_{+}$, $s \in I$ is a self-similar solution to the CF with  curvature $k(s)$ that corresponds to   a non trivial solution $\psi(s)=(\alpha(s),\tau(s),\eta(s))$  of (\ref{s2}), defined on the maximal interval $I=(\omega_{-}, \omega_{+})$, $a>0$, $c<0$ with  initial condition $\psi(0)\in H \cup C\cup S$. Then 

i) If $\psi(0)\in H\cup C$, then  $\omega_+=+\infty$, $ \lim_{s \to +\infty}k(s)=c$ and \;  either $\lim_{s \to -\infty}k(s)=0$ or $\lim_{s \to \omega_{-}}k(s)=+\infty.$	
	
ii) 	If $\psi(0)\in S$, then at each end of $X$ the curvature either converges to zero or it is unbounded. 
	\end{corollary}


 We will now study the solutions of \eqref{s2} when  $c>0$. In this case, there are no singular solutions  in $H\cup C$. In the next lemma we study the functions $\eta(s)$ and $\tau(s)$, when $c>0$ and  $\psi(0)\in H$.

\begin{lemma}\label{c41l8}Let $\psi(s)=(\alpha(s),\tau(s),\eta(s))$ be a non trivial solution of (\ref{s2}) defined on the maximal interval $I=(\omega_{-}, \omega_{+})$, $a>0$, $c>0$ and initial condition $\psi(0)\in H$. Then $\omega_{+}=+\infty$, $ \lim_{s \to +\infty} c+a\tau(s)=\lim_{s \to +\infty} \eta(s)=0$. Moreover, $ \lim_{s \to \omega_{-}} \tau(s)=\lim_{s \to \omega_{-}} \eta(s)=+\infty$.
\end{lemma}
\begin{proof}
 Lemma \ref{c40l3} implies that there exists a global maximum point $s_0$ of $\alpha(s)$. Thus, $\alpha(s)$ is a decreasing (resp. increasing) function on  $(s_0, \omega_{+})$ i.e. $\tau(s)=\alpha'(s)<0$ for all $s>s_0$ (resp. $( \omega_{-},s_0)$ i.e. $\tau(s)>0$ for all $s<s_0$ ). Moreover, from Lemma \ref{ln3} we have that $\tau(s)$ does not have critical point on $(\omega_{-},s_0)$ and it has at  most one critical point on $[s_0,\omega_+)$.  
	
\bf Claim. \it  There exists a unique $s_1>s_0$ such that $c+a\tau(s_1)=0$ and $c+a\tau(s)<0$ for all $s>s_1$. \rm In fact, assume by contradiction that there exists a solution  $\psi(s) \in S$, $s \in I$, of \eqref{s2} such that $c+a\tau(s)>0$ and $\tau(s)<0$ for all $s>s_0$. Then the third  equation of \eqref{s2} implies that the positive function $\eta(s)$ is increasing on $(s_0, \omega_{+})$.  It follows from Lemma \ref{c40l4extra} that $\alpha(s)$ is unbounded on $(s_0, \omega_{+})$. Since $\alpha(s)$ is an increasing function on $(s_0, \omega_{+})$ and $2\alpha(s)\eta(s)+\tau^{2}(s)=-1$, then $\lim_{s \to \omega_{+}} 2\alpha(s)\eta(s)=-\infty$ and $\lim_{s \to \omega_{+}} \tau^{2}(s)=+\infty$, which contradicts the fact that $c+a\tau(s)>0$ with $\tau(s)<0$ and $c>0$ for all $s>s_0$. Therefore, there exists $s_1 \in I$ such that $c+a\tau(s_1)=0$. Suppose that there exist $s_1, s_2 \in (s_0,\omega_{+})$ such that $c+a\tau(s_1)=c+a\tau(s_2)=0$ and $c+a\tau(s)\neq 0$ for all $s \in (s_1,s_2)$. Thus, $s_1$ and $s_2$ are consecutive critical points and local maximum points of $\eta(s)$, this is a contradiction. Therefore, $s_1$ is unique.  

Our Claim implies that $[c+a\tau(s)]\tau(s)>0$ for all $s>s_1$. Thus, \eqref{s2} implies that the positive function    $\eta(s)$ is decreasing on $(s_1,\omega_{+})$ i.e. $\eta(s)$ is bounded on $(s_1,\omega_{+})$. Hence, it follows from Lemma \ref{g3} that $ \lim_{s \to +\infty} c+a\tau(s)=\lim_{s \to +\infty} \eta(s)=0$.

Finally, since $\tau(s)>0$ and the functions $\alpha(s)$, $\tau(s)$ and $\eta(s)$ are monotone on $(\omega_{-},s_0)$, then it follows from \eqref{s2} that $\eta(s)$ is a decreasing positive function on $(\omega_{-},s_0)$. Hence, $ \lim_{s \to \omega_{-}} \eta(s)\alpha(s)=-\infty$. Therefore, $ \lim_{s \to \omega_{-}} \eta(s)=\lim_{s \to \omega_{-}} \tau(s)=+\infty.$
\end{proof}

\begin{lemma}\label{c41l9}
	Let $\psi(s)=(\alpha(s),\tau(s),\eta(s))$ be a non trivial solution of (\ref{s2}) defined on the maximal interval $I=(\omega_{-}, \omega_{+})$, $a>0$, $c>0$ and initial condition $\psi(0)\in C$. Then either $\lim_{s \to \omega_{-}}\eta(s)=\lim_{s \to \omega_{-}} c+a\tau(s)=+\infty$ or $\lim_{s \to -\infty} c+a\tau(s)=c$, and $\lim_{s \to +\infty} c+a\tau(s)=\beta$ where $\beta \in \{0,c\}$. Moreover, if $\alpha(s)$ is bounded on $(\omega_{-},\overline{s})$ (resp. $(\overline{s},\omega_{+})$), $\overline{s} \in I$ then $\lim_{s \to -\infty} \alpha^2(s)+\tau^{2}(s)+\eta^{2}(s)=0$ (resp. $\lim_{s \to +\infty} \alpha^2(s)+\tau^{2}(s)+\eta^{2}(s)=0$).
\end{lemma}

\begin{proof}
	Item \it i) \rm of  Lemma \ref{ln6} implies that there exist $s_1,s_2 \in I$ such that $\alpha(s)$, $\tau(s)$ and $\eta(s)$ are monotone functions on $(\omega_{-},s_1)$ and $(s_2, \omega_{+})$.
	 If $\alpha(s)$ is bounded on $(\omega_{-},s_1)$, then it follows from Lemma \ref{g2} that $\omega_{-}=-\infty$, $\lim_{s \to -\infty}\tau(s)=0$ and $\lim_{s \to -\infty}\tau^{\prime}(s)=0$ i.e. $\lim_{s \to -\infty} c+a\tau(s)=c$.  Lemma \ref{novo} implies that $\lim_{s \to -\infty} \alpha^{2}(s)=\lim_{s \to -\infty} \eta^{2}(s)=0$. 
	 If $\alpha(s)$ is unbounded on $(\omega_{-},s_1)$, then $\lim_{s \to \omega_{-}}\alpha(s)=-\infty$ and it is an increasing function on $(\omega_{-},s_1)$ i.e. $\tau(s)>0$ for all $s<s_1$. Thus, $[c+a\tau(s)]\tau(s)>0$ and $\eta(s)$ is a decreasing positive function  for all $s<s_1$. Since $2\alpha(s)\eta(s)+\tau^{2}(s)=1$ it follows  that $\tau(s)$ is unbounded on $(\omega_{-},s_1)$ and $\lim_{s \to \omega_{-}}c+a\tau(s)=+\infty$.
	
 If  $\alpha(s)$ is bounded on $(s_2,\omega_{+})$, similarly to the  previous case we can prove that  then $\omega_{+}=+\infty$ $, \lim_{s \to +\infty}\tau(s)=0$ and $\lim_{s \to -\infty} \alpha^{2}(s)=\lim_{s \to -\infty} \eta^{2}(s)=0$. 
	Now, suppose that $\alpha(s)$ is unbounded on $(s_2,\omega_{+})$, then $ \lim_{s \to \omega_{+}} \alpha(s)=-\infty$, $\tau(s)<0$ for all $s>s_2$. Assume by contradiction that $\tau(s)$ is unbounded on $(s_2,\omega_{+})$. Thus, there exists $s_3>s_2$ such that $[c+\tau(s)]\tau(s)>0$ for all $s>s_3$ and $\eta(s)$ is a positive decreasing function on $(s_3, \omega_{+})$ i.e. $\eta(s)$ is bounded on $(s_3, \omega_{+})$ which contradicts the third equation of \eqref{s2}, because $-\int_{s_3}^{s}[c+a\tau(s)]\tau(s)ds=\eta(s)-\eta(s_3)$. Hence, $\tau(s)$ is bounded on $(s_2,\omega_{+})$. Since $ \lim_{s \to \omega_{+}} \alpha(s)=-\infty$ then  equation $2\alpha(s)\eta(s)+\tau^{2}(s)=0$ implies that $\lim_{s \to \omega_{+}} \eta(s)=0$ and $\eta(s)$ is bounded on $(s_2,\omega_{+})$.  Therefore, Lemma \ref{g3} implies that $\lim_{s \to +\infty} [c+a\tau(s)]=0.$
\end{proof}

Finally, we can study the case $ \psi(0) \in S$. Note that, when $a=c$ then $\psi(s)=(-s,-1,0)$ is a trivial solution of $(\ref{s2})$ and it represents a parable in $Q^2_{+}$. i.e. a trivial solution to the CF.

\begin{lemma}\label{c41ld1}
	Consider  $\Phi:S\rightarrow TS \subset \R^{3}$,  given  by $\Phi(\alpha,\tau,\eta)=\left(\tau,c\alpha+a\tau\alpha-\eta,-c\tau-a\tau^{2}\right)$, 
	where $a,c>0$ and $S$ is given by \eqref{h}. Then $ p=\left(-{1}/{\sqrt{2c}},0,-{c}/{\sqrt{2c}}\right)$ and $-p$ are singular points of $\Phi$, and the eigenvalues of $d\Phi_{p}$ and $d\Phi_{-p}$ are
	\begin{equation} \label{la}
	\displaystyle \lambda_{p}=\frac{a\pm \sqrt{a^{2}+16c^{2}}}{2\sqrt{2c}} \hspace{1cm} \text{and} \hspace{1cm} \lambda_{-p}=\frac{-a\pm \sqrt{a^{2}+16c^{2}}}{2\sqrt{2c}}.
	\end{equation}
	\end{lemma}

\begin{proof} It follows from the fact that 
	the singular points of $\Phi$ satisfy  $\tau=0$, $c\alpha=\eta$ and  $1=2\alpha\eta+\tau^{2}$.
\end{proof}

The singular points of $\Phi$, $\pm p$, are saddle points. Each singular solution of \eqref{s2} with $c>0$ and $\psi(0) \in S$ corresponds to a hyperbole obtained by intersecting  the  light cone with  one of the planes $x_3\sqrt{2c}= \pm 1$.
Moreover,  there exist $q, \overline{q} \in S\setminus \{p\}$ such that $ \lim_{s \to +\infty}\psi(s,q)=\pm p $ and $ \lim_{s \to -\infty}\psi(s,\overline{q})=\pm p $. Define the stable and unstable sets
\begin{equation}\label{w2}
\begin{array}{l}
W^{s}(\pm p)=\{q \in S: \lim_{s \to +\infty}\psi(s,q)=\pm p\},\\
W^{u}(\pm p)=\{q\in S: \lim_{s \to -\infty}\psi(s,q)=\pm p\}.
\end{array}
\end{equation}

Our next lemma provides the behaviour of $\alpha(s)$  and $\tau(s)$   
at the boundary of the maximal interval of a non trivial solution 
of \eqref{s2}. 

\begin{lemma}\label{c41ld7d8}
Let $\psi(s)=(\alpha(s),\tau(s),\eta(s))$ be a non trivial solution of (\ref{s2}) defined on the maximal interval $I=(\omega_{-}, \omega_{+})$, $a>0$, $c>0$ and initial condition $\psi(0)\in S$. Consider $W^{u}(\pm p)$ and $W^{s}(\pm p)$ as in \eqref{w2}. 

i) If $\psi(0) \in S\setminus W^{u}(\pm p)$ (resp. $S\setminus W^{s}(\pm p)$) then $\lim_{s \to \omega_{-}}|\alpha(s)|=+\infty$ (resp. $ \lim_{s \to \omega_{+}}|\alpha(s)|=+\infty$). 

ii)Either $\lim_{s \to \omega_{-}} |c+a\tau(s)|=+\infty$ or $\lim_{s \to -\infty} c+a\tau(s)=\beta$ and  either $\lim_{s \to \omega_{+}} |c+a\tau(s)|=+\infty$ or $\lim_{s \to +\infty} c+a\tau(s)=\beta$, where the constant $\beta \in \{0,c\}$.
\end{lemma}

\begin{proof} \it i) \rm  The limits are obtained by using the same arguments of the  proof of Lemma \ref{c40l4extra}. In order to prove 	
\it ii) \rm we note that Item \it i) \rm of Lemma \ref{ln6}  implies  that there exist $s_1,s_2 \in I$ such that $\alpha(s)$, $\tau(s)$ and $\eta(s)$ are monotone on $(\omega_{-},s_1)$ and $(s_2, \omega_{+})$. 
 If $\psi(0) \in W^{u}(\pm p)$ (resp. $\psi(0) \in W^{s}(\pm p)$) then by definition $\omega_{-}=-\infty$, $\lim_{s \to -\infty} \tau(s)=0$ and $\lim_{s \to -\infty} c+a\tau(s)=c$ (resp. $\omega_{+}=+\infty$, $\lim_{s \to +\infty} \tau(s)=0$ and $\lim_{s \to +\infty} c+a\tau(s)=c$). 
   If $\psi(0) \in S\setminus W^{u}(\pm p)$ (resp. $S\setminus W^{s}(\pm p)$) then it follows from \it i) \rm that $\lim_{s \to \omega_{-}}|\alpha(s)|=+\infty$ (resp. $ \lim_{s \to \omega_{+}}|\alpha(s)|=+\infty$). Thus, Lemma \ref{g3} implies that if $ \psi(0) \in S\setminus W^{u}(\pm p)$ (resp. $ \psi(0) \in S\setminus W^{s}(\pm p)$) then either $\lim_{s \to \omega_{-}} |c+a\tau(s)|=+\infty$ or $\lim_{s \to -\infty} c+a\tau(s)=0$ (resp. either $\lim_{s \to \omega_{+}} |c+a\tau(s)|=+\infty$ or $\lim_{s \to +\infty} c+a\tau(s)=0$). 
\end{proof}

As a consequence of the previous lemma, we can determine the behaviour of the curvature function  at each end of a self-similar solution to the CF 
on  $Q^{2}_{+}$, when $c>0$.    

\begin{corollary}\label{coroc10_11_12}
Let  $X: I \rightarrow Q^{2}_{+}$, $s \in I$ be a curve which  is a self-similar solution to the CF, with curvature $k(s)$, that corresponds to 
a non trivial solution $\psi(s)=(\alpha(s),\tau(s),\eta(s))$ of (\ref{s2}) defined on the maximal interval $I=(\omega_{-}, \omega_{+})$, $a>0$, $c>0$ and initial condition $\psi(0)\in H\cup C\cup S$. Then  
	
i) If $\psi(0)\in H$, then   $\omega_+=+\infty$, $\lim_{s \to \omega_{-}}k(s)=+\infty$ and $ \lim_{s \to +\infty}k(s)=0$.		

ii) If 	$\psi(0)\in C\cup S$ then at each end of $X$ either $k(s)$ is unbounded or it tends to $0$ or $c$. 
\end{corollary}

\begin{proof} Since $k(s)=c+a\tau(s)$, $a>0$ $c>0$, 
{\it i)} and  {\it ii)} follow from  Lemmas \ref{c41l8}, \ref{c41l9} and \ref{c41ld7d8}. 
\end{proof}

We can now prove our main result

\begin{proof}[\bf Proof of Theorem \ref{c3t4}\rm]
 Consider any vector $v\in \mathbb{R}^3\setminus\{0\}$.  Without loss of generality we can assume that $v=ae$, where $a>0$ and 
	\begin{eqnarray*}
		e=	\left\{\begin{array}{l}
			(1,0,0)\hspace{0.3 cm} \text{if} \hspace{0.3 cm}v \hspace{0.3 cm} \text{is a timelike vector},\\
			(1,1,0)\hspace{0.3 cm} \text{if} \hspace{0.3 cm} v \hspace{0.3 cm}  \text{is a lightlike vector},\\
			(0,0,1)\hspace{0.3 cm} \text{if} \hspace{0.3 cm} v \hspace{0.3 cm}  \text{is a spacelike vector.}
		\end{array} \right.
	\end{eqnarray*}	 
	Let $\psi(s)=(\alpha(s),\tau(s),\eta(s))$ be a solution of the system \eqref{s2},  with  $c\in\R$,  defined on a maximal interval $I$ and initial condition $\psi(0) \in \R^{3}$ such that 		 
	\begin{eqnarray*}
		2\alpha(0)\eta(0)+\tau^{2}(0)=	\left\{\begin{array}{l}
			-1\hspace{0.3 cm} \text{if} \hspace{0.3 cm}v \hspace{0.3 cm} \text{is a timelike vector},\\
			0 \hspace{0.3 cm} \text{if} \hspace{0.3 cm} v \hspace{0.3 cm}  \text{is a lightlike vector},\\
			1 \hspace{0.3 cm} \text{if} \hspace{0.3 cm} v \hspace{0.3 cm}  \text{is a spacelike vector,}
		\end{array} \right.
	\end{eqnarray*}
It follows from Proposition \ref{c4p2} that, for each initial condition,  there is a non trivial self-similar  solution  to the CF,  $X(s)$ in $Q^2_+$, with curvature $k(s)=c+ a\tau(s)$. Thus, the initial conditions of $\alpha,\tau$ and $\eta$,  which are given by  two constants,  determine the self-similar solution to the CF. Therefore, for each fixed vector $v \in \R_{1}^{3}\setminus \{0\}$ and each  constant  $c\in\R$, there is a 2-parameter family of non trivial self-similar solutions to the CF in ${Q}^{2}_+$. These curves are solitons when $c=0$. Moreover, there are three    classes of such solutions corresponding to the three types of vectors $v$.
	
	 It follows from  item \it iii) \rm of Lemma \ref{ln4} that the curvature function $k(s)$ has at  most two zeros.  Moreover, 
Corollaries  \ref{coroc3_4}, \ref{coroc5_6_9} and \ref{coroc10_11_12}	
show that, at each end, the curvature function is unbounded or it tends to one of the two constants $\{ c,0\}$.
 
Let $Y(s)$ be the associated curve to $X(s)$.  When $k(s)\neq 0$ it follows from Remark \ref{c3t0} and Proposition \ref{c3p1} that $-Y(s)$ is a self-similar solution to the ICF on $Q_{+}^{2}$. Moreover, if $s_1,s_2 \in I$ are the zeros (resp. $s_1$ is the zero) of $k(s)$ on $I$, then $Y_1=-Y|_{(\omega_{-},s_1)}$, $Y_2=-Y|_{((s_1,s_2)}$ and $Y_3=-Y|_{(s_2, \omega_{+})}$ (resp. $Y_1=-Y|_{(\omega_{-},s_1)}$ and $Y_2=-Y|_{((s_1,\omega_{+})}$) are self-similar solutions to the ICF on $Q_{+}^{2}$.  Therefore, given $v \in \R_{1}^{3}\setminus \{0\}$ and $c \in \R$, there exist  a 2-parameter family of self-similar solutions to the ICF in $Q_{+}^{2}$. Any such curve may have at most three connected components. Since $\tilde{k}(s)=k^{-1}(s)$ is the curvature of $-Y(s)$ then at each end of $-Y(s)$ the curvature $\tilde{k}(s)$ is unbounded or it tends to one of the following constants $\{1/c,0\}$.  
\end{proof}
	
We conclude this section by providing explicit soliton solutions to the CF and to the ICF. They are obtained by considering $c=0$ and $v$ a light like vector. 

\begin{proposition}\label{c4ld1}
	Let $\tilde{X}: I \subset \R \rightarrow Q_{+}^{2}$ defined by $\tilde{X}(s)=(\tilde{x}_1(s),\tilde{x}_2(s),\tilde{x}_3(s))$, $s \in I$ be a soliton solution to the ICF in $Q^{2}_{+}$ that corresponds the solution $\tilde{\psi}(s)=(\tilde{\alpha}(s),\tilde{\tau}(s),\tilde{\eta}(s))$ of \eqref{s2} with $\tilde{\psi}(0) \in C$. Then $-\tilde{Y}(s)$ is a soliton solution to the CF in $Q_{+}^{2}$, where $\tilde{Y}(s)=(\tilde{y}_1(s),\tilde{y}_2(s),\tilde{y}_3(s))$ is the associated curve of $\tilde{X}(s)$, and for $s$ such that  $s<a\tilde{\eta}_0 $, $\tilde{n}_0=\tilde{\eta}(0)>0$,  we have
	\begin{equation*}\label{ws}
	\tilde{X}(s)=\left(-\frac{\tilde{x}^{2}_3(s)+\tilde{\alpha}^{2}(s)}{2\tilde{\alpha}(s)}, \frac{-\tilde{x}^{2}_3(s)+\tilde{\alpha}^{2}(s)}{2\tilde{\alpha}(s)},\tilde{x}_3(s)    \right),  \quad  \tilde{\alpha}(s)=-\frac{\tilde{\tau}^{2}(s)}{2\tilde{\eta}(s)}, \quad \tilde{\eta}(s)=-\frac{s}{a}+\tilde{\eta}_0,
	\end{equation*}
\begin{equation*}
\tilde{Y}(s)= \left(-\frac{\tilde{y}^{2}_3(s)+\tilde{\eta}^{2}(s)}{2\tilde{\eta}(s)}, \frac{-\tilde{y}^{2}_3(s)+\tilde{\eta}^{2}(s)}{2\tilde{\eta}(s)}  ,\tilde{y}_3(s)  \right), \quad  \tilde{\tau}(s)=\frac{2a\tilde{\eta}^{2}(s)}{3}+\sqrt{\tilde{\eta}(s)}\left(\frac{3\tilde{\tau}_0-2a\tilde{\eta}^{2}_0}{3\sqrt{\tilde{\eta}_0}}\right),
\end{equation*}
\begin{equation*}
	\tilde{y}_{3}(s)=\frac{\tilde{\eta}(s)\tilde{x}_{3}(s)-\tilde{\tau}(s)}{\tilde{\alpha}(s)}, \quad \tilde{x}_{3}(s)=\left(\tilde{\eta}^{3/2}(s)+D\right)^{2}\left(-\int{(\tilde{\eta}^{3/2}(s)+D)^{-2}}ds+C\right)  
\end{equation*}
and $\displaystyle D=\frac{3}{2a}\left(\frac{3\tilde{\tau}_0-2a\tilde{\eta}^{2}_0}{3\sqrt{\tilde{\eta}_0}}\right).$ Moreover, if $D>0$
		\begin{equation*}
			\int{(\tilde{\eta}^{3/2}(s)+D)^{-2}}ds=
\frac{1} { 9{D}^{\frac{4}{3}} }\log 
\frac{ \tilde{\eta}(s)- D^{\frac{1}{3}} \left(  \sqrt {\tilde{\eta}(s)} -D^{\frac{1}{3}}  \right)}{\left(\sqrt{\tilde{\eta}(s)}+D^{\frac{1}{3}} \right)^2} +	 		
\end{equation*}	
\begin{equation}\label{D}
\hspace{3.4 cm} +\frac{2}{3\sqrt{3}D^{\frac{4}{3}}}\left[
\arctan\left(\frac{2\sqrt{\tilde{\eta}(s)}-D^{\frac{1}{3}}}{\sqrt{3}D^{\frac{1}{3}}}\right)+\frac{\sqrt{3}\,\tilde{\eta}(s)}
{ D^{-\frac{1}{3}} \tilde{\eta}(s)^{\frac{3}{2}}+D^\frac{2}{3}}\right].
\end{equation}
\end{proposition}
\begin{proof}
Suppose that $\tilde{X}(s)$ is a soliton solution to the ICF, i.e. $c=0$.
It follows from system \eqref{s3}, with $c=0$, that 
$\tilde{\eta}(s)=-{s}/{a}+\tilde{\eta}_0$. Let  
$ r(s)=\tilde{\eta}(s)=-{s}/{a}+\tilde{\eta}_0>0$,
$r_0=\tilde{\eta}_0>0$. Then $$
\frac{dr}{ds}=-\frac{1}{a},\hspace{0.3 cm} \displaystyle
\frac{d\tilde{\alpha}}{ds}=-\frac{1}{a}\frac{d\tilde{\alpha}}{dr},\hspace{0.3
cm} \displaystyle
\frac{d\tilde{\tau}}{ds}=-\frac{1}{a}\frac{d\tilde{\tau}}{dr},\hspace{0.3
cm} \displaystyle
\frac{d\tilde{\eta}}{ds}=-\frac{1}{a}\frac{d\tilde{\eta}}{dr} \quad
\text{and} \quad \left\{\begin{array}{ll}
\tilde{\alpha}^{\prime}(r)=-a\tilde{\tau}(r),\\
\displaystyle
\tau^{\prime}(r)=-\frac{\tilde{\alpha}(r)}{\tilde{\tau}(r)}+ar.
\end{array}\right.$$
Using that $2r\tilde{\alpha}(r)+\tilde{\tau}^{2}(r)=0$ we obtain
$\displaystyle
-\frac{\tilde{\alpha}(r)}{\tilde{\tau}(r)}=\frac{\tilde{\tau}(r)}{2r}$ and
$\displaystyle \tilde{\tau}^{\prime}(r)-\frac{\tilde{\tau}(r)}{2r}=ar$.
Taking $\tilde{\tau}(0)=\tau_0$ we have
\begin{equation}\label{t}
\tilde{\tau}(r)=\frac{2ar^{2}}{3}+r^{\frac{1}{2}}C_1=\frac{2ar^{2}}{3}+r^{\frac{1}{2}}\left(\frac{3\tilde{\tau}_0-2a\tilde{\eta}^{2}_0}{3\sqrt{\tilde{\eta}_0}}\right).
\end{equation}
Since $\tilde{\alpha}(s)=-\tilde{x}_1(s)+\tilde{x}_2(s)$,
$\tilde{\eta}(s)=-\tilde{y}_1(s)+\tilde{y}_2(s)$ and $\tilde{X}(s)$ and
$\tilde{Y}(s)$ are curves in $Q^{2}_{+}$ then
$-\tilde{x}^{2}_1(s)+[\tilde{\alpha}(s)+\tilde{x}_1(s)]^{2}+\tilde{x}^{2}_3(s)=0
$,
 $-\tilde{y}^{2}_1(s)+[\tilde{\eta}(s)+\tilde{y}_1(s)]^{2}+\tilde{y}^{2}_3(s)=0$.
Thus,
$$
\displaystyle
2\tilde{\eta}(s)\tilde{y}_1(s)=-\tilde{y}^{2}_3(s)-\tilde{\eta}^{2}(s),\,\,
2\tilde{\eta}(s)\tilde{y}_2(s)=\tilde{\eta}^{2}(s)-\tilde{y}^{2}_3(s), \,\,
\displaystyle
2\tilde{\alpha}(s)\tilde{x}_1(s)=-\tilde{x}^{2}_3(s)-\tilde{\alpha}^{2}(s)\,\,$$
and $\displaystyle
2\tilde{\alpha}(s)\tilde{x}_2(s)=\tilde{\alpha}^{2}(s)-\tilde{x}^{2}_3(s).$
Now, we will determine the functions $\tilde{x}_3(s)$ and $\tilde{y}_3(s)$.
We know that $\tilde{T}(s)=\tilde{X}(s)\times \tilde{Y}(s)$, where
$\tilde{T}(s)$ is the unit tangent vector field and $\tilde{Y}(s)$ is the
associated curve to $\tilde{X}$. Hence,
\begin{eqnarray*}
\tilde{\tau}(s)=\spn{\tilde{T}(s),(1,1,0)}=-\tilde{x}^{\prime}_{1}(s)+\tilde{x}^{\prime}_{2}(s)
=
[-\tilde{x}_{1}(s)+\tilde{x}_{2}(s)]\tilde{y}_{3}(s)-[-\tilde{y}_{1}(s)+\tilde{y}_{2}(s)]\tilde{x}_{3}(s).
\end{eqnarray*}
Thus, $\tilde{\tau}(s)
=\tilde{\alpha}(s)\tilde{y}_{3}(s)-\tilde{\eta}(s)\tilde{x}_{3}(s)$ and
$\displaystyle
\tilde{\alpha}(s)\tilde{y}_{3}(s)=\tilde{\eta}(s)\tilde{x}_{3}(s)+\tilde{\tau}(s)$.
The equation $
\tilde{\eta}(s)\tilde{X}(s)+\tilde{\tau}(s)\tilde{T}(s)+\tilde{\alpha}(s)\tilde{Y}(s)=(1,1,0)$
implies that
$\tilde{\eta}(s)\tilde{x}_{3}(s)+\tilde{\tau}(s)\tilde{x}^{\prime}_{3}(s)+\tilde{\alpha}(s)\tilde{y}_{3}(s)=0$.
Hence,
$2\tilde{\eta}(s)\tilde{x}_{3}(s)+\tilde{\tau}(s)\tilde{x}^{\prime}_{3}(s)+\tilde{\tau}(s)=0$
and
\begin{equation*}
\tilde{x}^{\prime}_{3}(s)+2\frac{\tilde{\eta(s)}}{\tilde{\tau}(s)}\tilde{x}_{3}(s)=-1.
\end{equation*}
Taking $r(s)=\tilde{\eta}(s)$ i.e. $\displaystyle
\frac{d\tilde{x}_{3}}{ds}=-\frac{1}{a}\frac{d\tilde{x}_{3}}{dr}$  and using
\eqref{t}, we obtain
${\tilde{x}^{\prime}_{3}(r)-3\frac{r^{\frac{1}{2}}}
{r^{\frac{3}{2}}+\frac{3}{2a}C_{1}}}\tilde{x}_{3}(r)=a.$  
Thus,
\begin{equation*}
\tilde{x}_{3}(r)=\frac{1}{\mu(r)}\left(a\int{\mu(r)}dr+C\right),\quad
\text{where}\quad \mu(r)=\left(r^{3/2}+\frac{3}{2a}C_{1}\right)^{-2}.
\end{equation*}
Taking $\displaystyle 2aD= 3C_{1}$ we have
$\tilde{x}_{3}(r)=\left(r^{3/2}+D\right)^{2}\left(a\int{(r^{3/2}+D)^{-2}}dr+C\right).
$
If $D>0$ we obtain \eqref{D}.
\end{proof}

\section{Visualizing some Self-Similar Solutions to the CF and ICF}	
In this section, we visualize some examples of self-similar solutions to the CF and ICF on $Q_{+}^{2}$. We consider $Q_{+}^{2}$ parametrized by  $\chi(\rho,\varphi)=(\rho,\rho\, \cos\varphi,\rho\, \sin \varphi),\quad \rho>0$. 
If a curve  $X(s)=\chi(\rho(s),\varphi(s))$ is parametrized by arc length
$s\in I$, then $[\rho(s)\varphi^{\prime}(s)]^{2}=1.$ Taking $\rho(s)\varphi^{\prime}(s)=1$, the unit tangent vector field is given by  $T=\left(\rho^{\prime},\rho^{\prime}\cos\varphi)-\sin \varphi,\rho^{\prime}\sin \varphi+\cos \varphi\right)$. Moreover, 
the functions $\rho(s)$, $\varphi(s)$ and the curvature function $k(s)$ satisfy the following system of  ordinary differential equations  
\begin{equation}\label{ex2}
\left\{\begin{array}{ll}
\displaystyle \rho^{\prime \prime}=\rho k+\frac{1+[\rho^{\prime}]^{2}}{2\rho},\\
\varphi^{\prime}\rho=1.
\end{array}\right.		
\end{equation}
The curve $Y(s)$ associated to $X(s)$ can be written as 
\begin{eqnarray*}
	Y=\frac{1}{2\rho}\left(-{1+[\rho^{\prime}]^{2}},{2\rho^{\prime}\sin \varphi+ [1-(\rho^{\prime})^{2}]\cos \varphi }, {-2\rho^{\prime}\cos \varphi + [1-(\rho^{\prime})^{2}]\sin \varphi } \right).
\end{eqnarray*}

It follows from Theorem \ref{c3t2} that  $X(s)$, is a non trivial self-similar solution to the CF if and only if  there exist $v \in \R_{1}^{3}$, $c,a \in \R$ with $a>0$ such that $k(s)= c+a\spn{T(s),v}$. Whenever $k(s)\neq 0$ it follows from Remark \ref{c3t0} and Proposition \ref{c3p1} that $-Y(s)$ is a self-similar solution to the ICF. Since $k(s)$ has at  most two zeros
(see Theorem \ref{c3t4}) it follows that $-Y(s)$ may have at most three   connected components which are self-similar solutions to the ICF. 
 In what follows, we use \eqref{ex2} and the software \it Maple \rm to plot $X(s)$ and $-Y(s)$ on $Q_{+}^{2}$, for several choices of $c,\, a$
and $v$.
In Figures 1-3, we visualize soliton solutions  ($c=0$) $X$ to the CF and $-Y$ solutions to the ICF on $Q_{+}^{2}$, when $v$ is a timelike, lightlike and spacelike vector, respectively. Similarly, in Figures 4-7 we visualize  self-similar solutions $X$ to the CF and $-Y$ solutions to the ICF for $c<0$ and in Figures 8-10 for $c>0$.

By considering  $a=1,\, c=0$ and the timelike vector $v=(1,0,0)$, in Figure \eqref{f1a}, we visualize a soliton solution $X$ to the CF and 
  Figure \eqref{f1b} shows the graph of its curvature. It follows from Lemmas \ref{c40l3} and  \ref{c40l5} that $k(s)$ has a unique zero and  $\lim_{s \to +\infty}k(s)=0$. The curve $-Y$ associated to $X$ has two connected components visualized in Figures \eqref{f1c} and \eqref{f1d}, which are  soliton solutions to the ICF. 
\begin{figure}[h!]
	\centering	
	\subfloat[]{\includegraphics[height=2.3 cm]{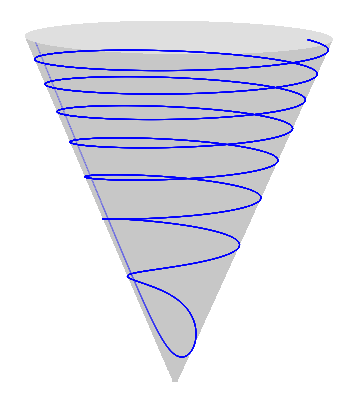}
\label{f1a}
} 
\quad
\subfloat[]{
\includegraphics[height=2.3 cm]{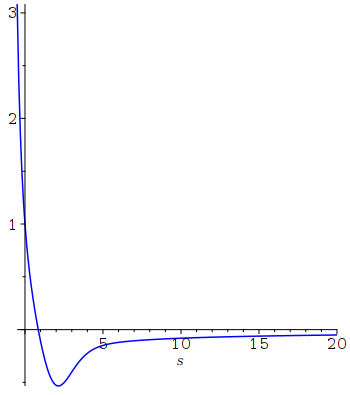}
\label{f1b}
}
\quad	
	\subfloat[]{
		\includegraphics[height=2.3 cm]{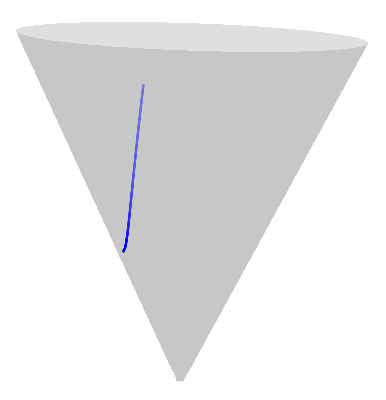}
\label{f1c}
	}
	\quad
	\subfloat[]{
		\includegraphics[height=2.3 cm]{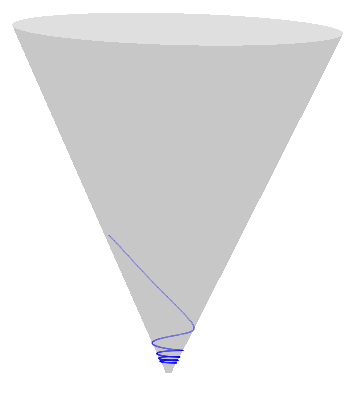}
		\label{f1d}
	}
	\caption{Soliton solutions (a) to the CF and (c) (d) to ICF on $Q^{2}_{+}$ with  $c=0$, $a=1$, $v=(1,0,0)$.}
	\label{fig1}
\end{figure} 

Proposition \ref{c4ld1} provides explicit soliton solutions $-\tilde{Y}$ to the CF  and $\tilde{X}$ to the ICF. Considering $c=0$, $a=1$ and the lightlike vector $v=(1,1,0)$, one can choose the initial conditions  so that the constant $D=0$. The soliton solutions $\tilde{Y}(s)$ and $\tilde{X}(s)$ are defined for $s<1$ and they are visualized in Figures \eqref{f2a} and \eqref{f2c} respectively.  Figure \ref{f2b} provides the graph of the curvature of $\tilde{Y}(s)$.
\begin{figure}[h!]
	\centering
	\subfloat[]{\includegraphics[height=2.3 cm]{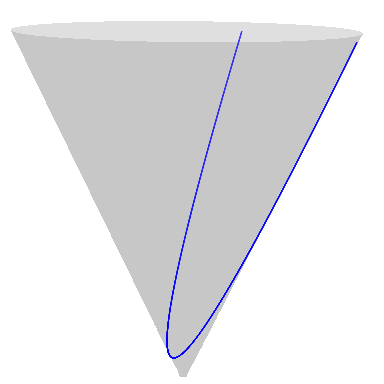}
		\label{f2a}
	}
	\subfloat[]{\includegraphics[height=2.2 cm]{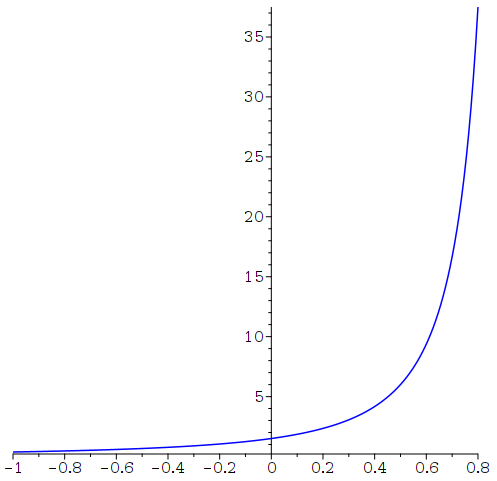}
		\label{f2b}
	}
	\subfloat[]{\includegraphics[height=2.3 cm]{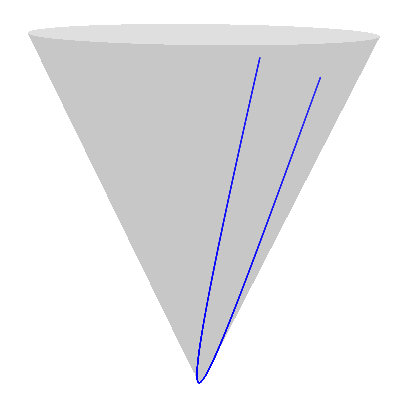}
		\label{f2c}
	}
	\caption{Soliton solutions (a) to the CF and (c) to ICF on $Q^{2}_{+}$ with $c=0$, $a=1$ and $v=(1,1,0)$.}
	\label{fig2}
\end{figure} 

In Figure \ref{f3a}, we visualize a soliton solution $X$ to the CF, by choosing  $c=0$, $a=1$ and the spacelike vector $v=(0,0,1)$. Figure \ref{f3b} shows the graph of $k(s)$. Since the curvature vanishes at two points, the curve $-Y$ associated to $X$ has three connected  components which are solutions to the ICF and they are visualized in  Figures \ref{f3c}, \ref{f3d} and \ref{f3e}.
\begin{figure}[h!]
	\centering
	\subfloat[]{\includegraphics[height=2.2 cm]{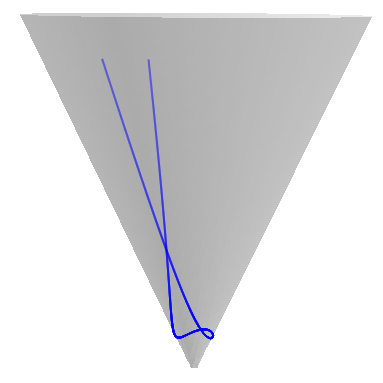}
		\label{f3a}
	}
\subfloat[]{\includegraphics[height=2.2 cm]{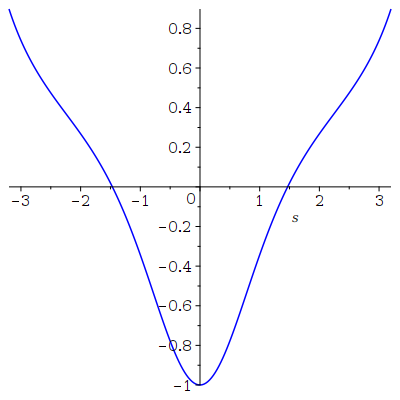}
	\label{f3b}
}
\subfloat[]{\includegraphics[height=2.2 cm]{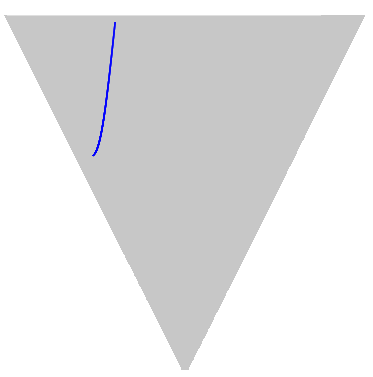}
	\label{f3c}
}
\subfloat[]{\includegraphics[height=2.2 cm]{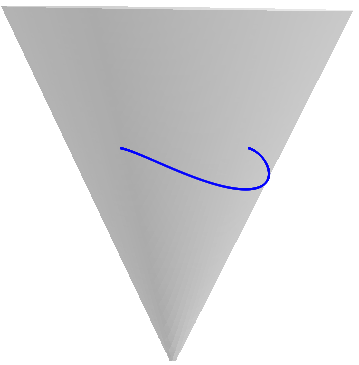}
	\label{f3d}
}
\subfloat[]{\includegraphics[height=2.2 cm]{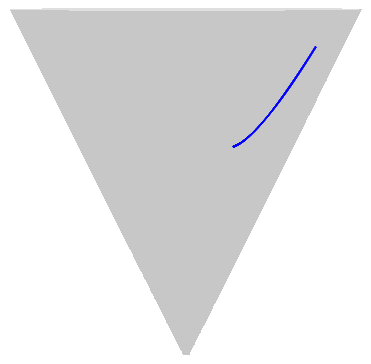}
	\label{f3e}
}
\caption{Soliton solutions  to the CF and  ICF on $Q^{2}_{+}$ with $c=0$, $a=1$, $v=a(0,0,1)$.}
	\label{fig3}
\end{figure} 

When $c<0$ and $v=a(1,0,0)$ is a timelike vector, then Lemma \ref{c4-1l1} 
implies that $c$ determines the singular point $p$ of $\Phi$ and  
the  associated eigenvalues are real  if $4|c|\leq a$ and complex if $4|c|>a$. In Figures 4 and 5, we visualize the self-similar solutions to the CF and to the ICF in each case.  
  Thus, $\psi(s)=p$, $s\in I$ is the singular solution of \eqref{s2} and it corresponds to a circle with radius $\displaystyle -1/(2c)$ in $Q^{2}_{+}$, and its associated curve is a circle with radius  
 $ -c/2$ in $Q^{2}_{-}$.  
 Moreover,  Lemma \ref{c4-1l6} implies that $ \lim_{s \to +\infty}k(s)=c$ and $\psi(s)=p$, $s \in I$ is a 
global attractor solution of \eqref{s2} in the set $H$.

We now consider self-similar solution with $c<0$. In Figure \ref{f4a}, we visualize a self-similar solution $X$ to the CF on $Q_{+}^{2}$ for  $ c=-{1}/{4}$,  $ a={1}/{6}$ and $v=a(1,0,0)$ is a timelike vector. Figure \ref{f4b} shows the graph of $k(s)$, which vanishes at one point. In Figure \ref{f4c}, we visualize one component of the self-similar solution  to the ICF on $Q_{+}^{2}$, associated to $X$,  that corresponds to $k(s)<0$.
\begin{figure}[h!]
	\centering
	\subfloat[]{\includegraphics[height=2.8 cm]{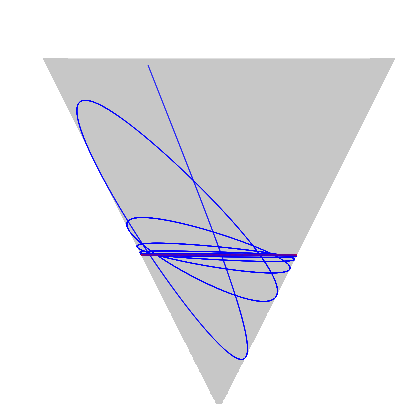}
		\label{f4a}
	}\quad
	\subfloat[]{\includegraphics[height=2.5 cm]{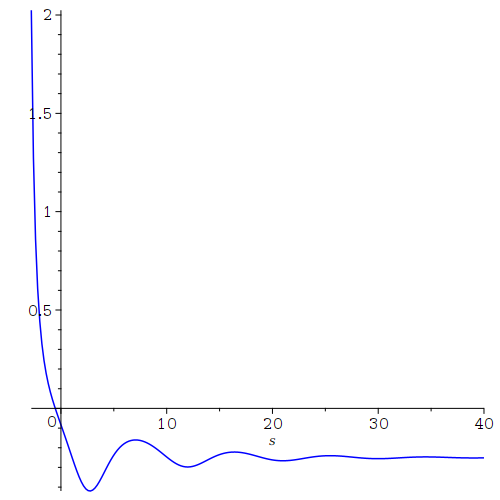}
		\label{f4b}
	}\quad
	\subfloat[]{\includegraphics[height=2.6 cm]{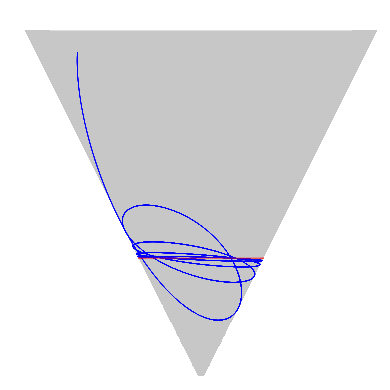}
		\label{f4c}
	}
	\caption{Self-similar solutions to the CF and  ICF on $Q^{2}_{+}$ with $c=-\frac{1}{4}$, $a=\frac{1}{6}$, $v=a(1,0,0)$.}
	\label{fig4}
\end{figure} 

In Figure \ref{f5a}, we visualize a self-similar solution $X$ to the CF on $Q_{+}^{2}$ for  $ c=-{1}/{4}$, $a=5$ and $v=a(1,0,0)$ is a timelike vector.  Figure \ref{f5b} shows the graph of $k(s)$, which vanishes at one point. In Figure \ref{f5c} we visualize one component of the  self-similar solution to the ICF on $Q_{+}^{2}$ that corresponds to $k(s)<0$.
\begin{figure}[h!]
	\centering
	\subfloat[]{\includegraphics[height=2.8 cm]{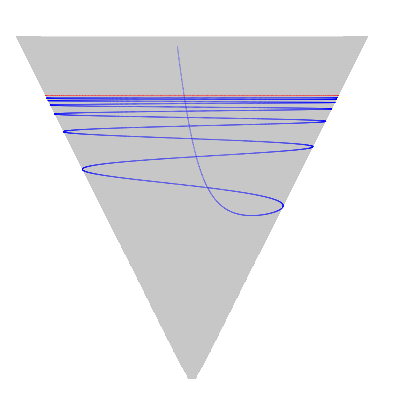}
		\label{f5a}
	}\quad
	\subfloat[]{\includegraphics[height=2.8 cm]{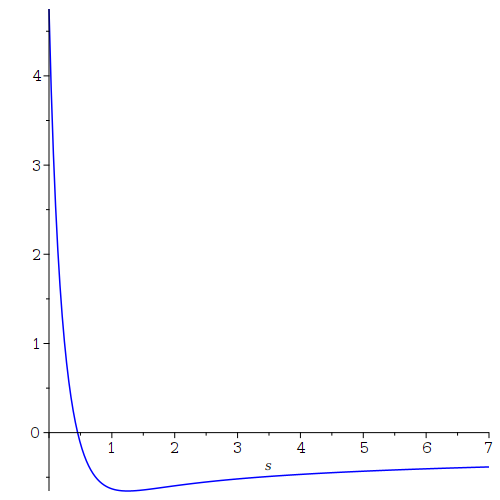}
		\label{f5b}
	}\quad
	\subfloat[]{\includegraphics[height=2.8 cm]{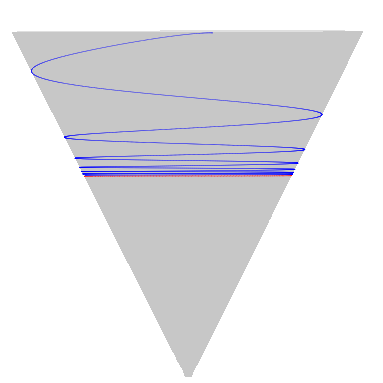}
		\label{f5c}
	}
	\caption{Self-similar solutions to the CF and ICF on $Q^{2}_{+}$ with $c=-\frac{1}{4}$, $a=5$ and $v=a(1,0,0)$.}
	\label{fig5}
\end{figure} 

In Figure \ref{f6a}, we visualize a self-similar solution $X$ to the CF in $Q_{+}^{2}$ when $ c=-2$, $a={1}/{2}$ and $v=a(1,1,0)$ is a lightlike vector. Figure \ref{f6b} shows the graph of $k(s)$, which vanishes at one point. It follows from Lemma \ref{c4-1l6} that $ \lim_{s \to +\infty}k(s)=c=-2$. In Figure \ref{f6c},  we visualize 
one component of the self-similar solution to the ICF on $Q_{+}^{2}$, associated to $X$, that corresponds to  $k(s)<0$.
\begin{figure}[h!]
	\centering
	\subfloat[]{\includegraphics[height=2.8 cm]{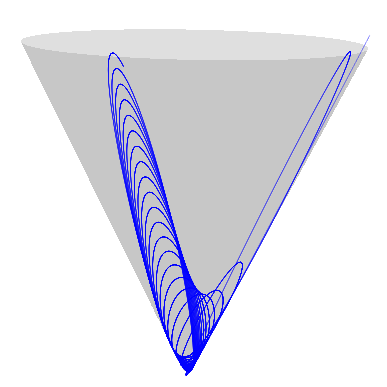}
		\label{f6a}
	}\quad
	\subfloat[]{\includegraphics[height=2.8 cm]{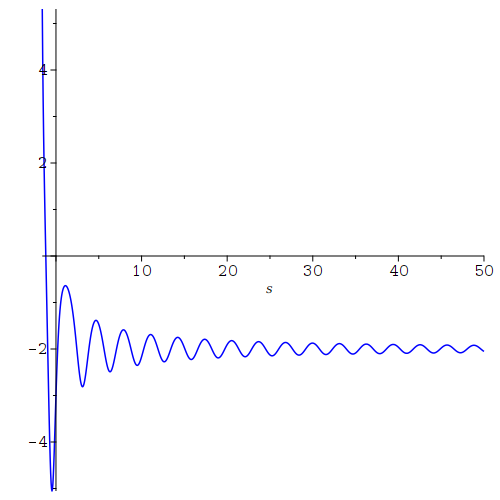}
		\label{f6b}
	}\quad
	\subfloat[]{\includegraphics[height=2.8 cm]{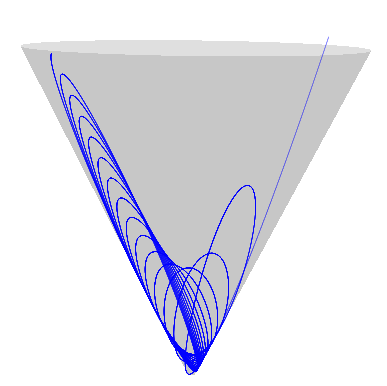}
		\label{f6c}
	}
	\caption{Self-similar solutions to the CF and ICF on $Q^{2}_{+}$ with $c=-2$, $a=\frac{1}{2}$ and $v=a(1,1,0)$.}
	\label{fig6}
\end{figure} 

In Figure \ref{f7a}, we visualize a self-similar solution $X$ to the CF on $Q_{+}^{2}$ when  $ c=-1$,  $ a=1$ and $v=(0,0,1)$ is a spacelike vector.  Figure \ref{f7b} shows the graph of $k(s)$ which vanishes at two points. In Figures \ref{f7c}, \ref{f7d} and \ref{f7e},  we visualize the three connected components of the self-similar solution to the ICF on $Q_{+}^{2}$, associated to $X$. 
\begin{figure}[h!]
	\centering
	\subfloat[]{\includegraphics[height=2.0 cm]{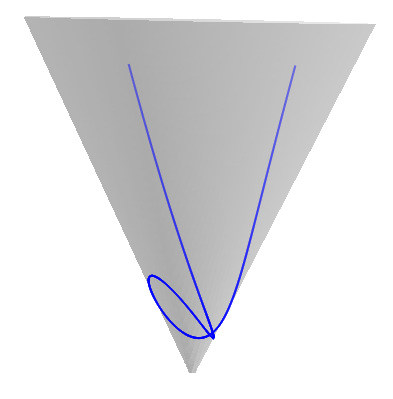}
		\label{f7a}
	}\quad
	\subfloat[]{\includegraphics[height=1.9 cm]{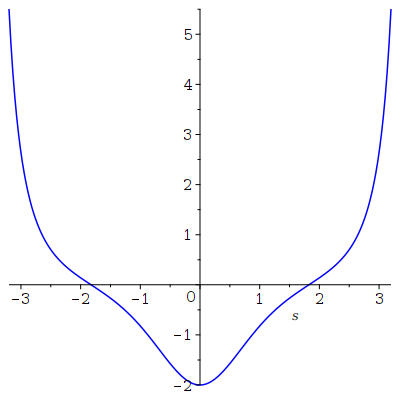}
		\label{f7b}
	}\quad
	\subfloat[]{\includegraphics[height=2.0 cm]{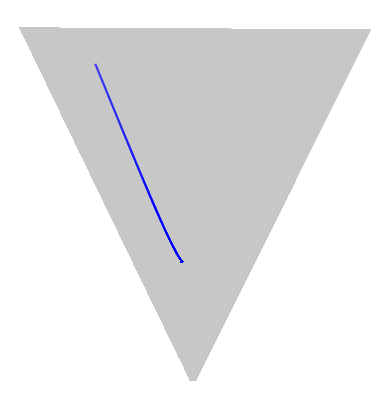}
		\label{f7c}
	}
\subfloat[]{\includegraphics[height=1.9 cm]{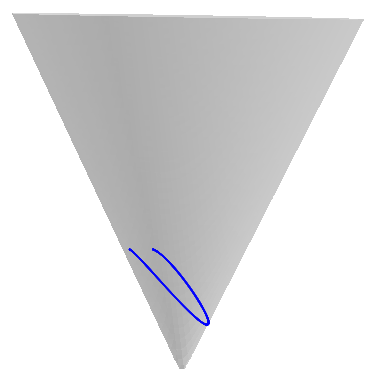}
	\label{f7d}
}\subfloat[]{\includegraphics[height=1.9 cm]{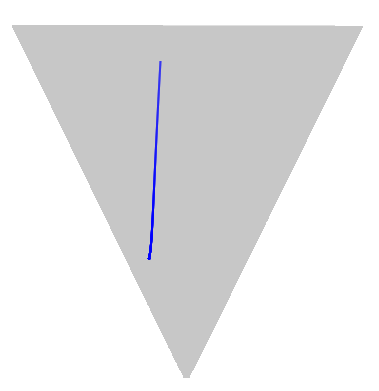}
	\label{f7e}}
\caption{Self-similar solutions to the CF and ICF on $Q^{2}_{+}$ with $c=-1$, $a=1$ and $v=(0,0,1)$.}
	\label{fig7}
\end{figure} 
We conclude considering  self-similar solutions with $c>0$.  In Figure \ref{f8a}, we visualize a self-similar solution $X$ to the CF on $Q_{+}^{2}$ when  $ c=4$, $ a=1$ and $v=(1,0,0)$ is a timelike vector.  Figure \ref{f8b} shows the graph of $k(s)$, which vanishes at one point.  Lemma \ref{c41l8} implies that $ \lim_{s \to +\infty}k(s)=0.$ Finally, in Figures \ref{f8c} and \ref{f8d} we visualize the two components of the  self-similar solution to the ICF on $Q_{+}^{2}$ associated to $X$.
\begin{figure}[h!]
	\centering
	\subfloat[]{\includegraphics[height=2.1 cm]{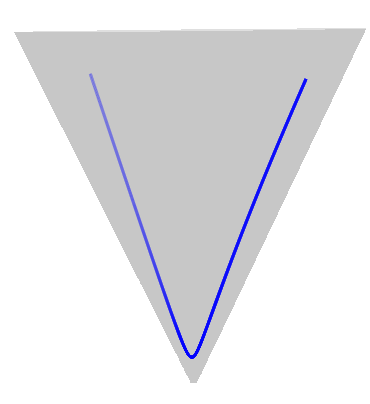}
		\label{f8a}
	}\quad
	\subfloat[]{\includegraphics[height=2.0 cm]{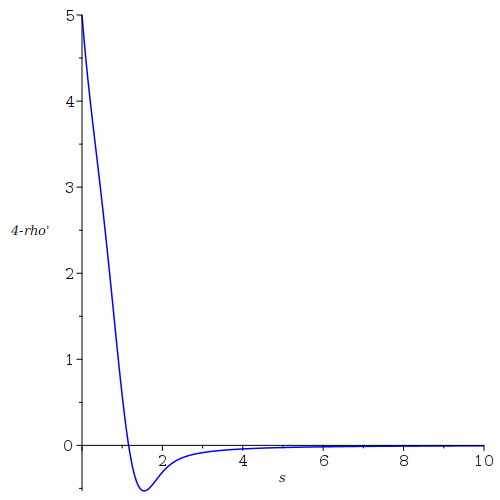}
		\label{f8b}
	}\quad
	\subfloat[]{\includegraphics[height=2.0 cm]{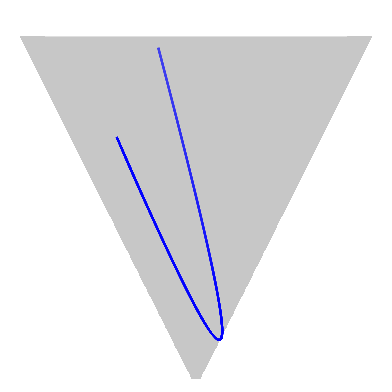}
		\label{f8c}
	}
\subfloat[]{\includegraphics[height=2.1 cm]{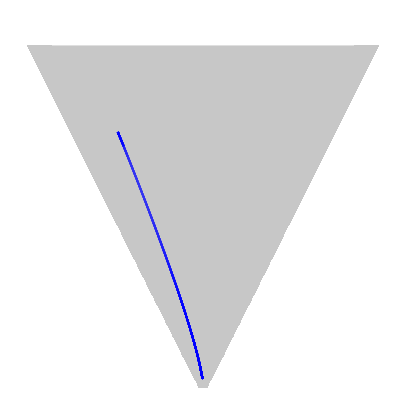}
	\label{f8d}}
	\caption{Self-similar solutions to the CF and ICF on $Q^{2}_{+}$ with $c=4$, $a=1$ and $v=(1,0,0)$.}
	\label{fig8}
\end{figure} 
In Figure \ref{f9a}, we visualize a self-similar solution $X$ to the CF on $Q_{+}^{2}$ when  $c=4$,  $a=1$  and $v=(1,1,0)$ is a lightlike vector.  Figure \ref{f9b} shows the graph of $k(s)$, which vanishes at one point.  Lemma \ref{c41l9} implies that when $s$ tends to $+\infty$ then $k(s)$ tends to $0$ or $c$.  In Figures \ref{f9c} and \ref{f9d} we visualize the two components of the self-similar solutions to the ICF on $Q_{+}^{2}$ associated to $X$.
\begin{figure}[h!]
	\centering
	\subfloat[]{\includegraphics[height=2.0 cm]{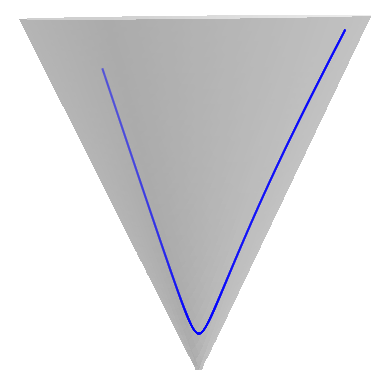}
		\label{f9a}
	}\quad
	\subfloat[]{\includegraphics[height=2.0 cm]{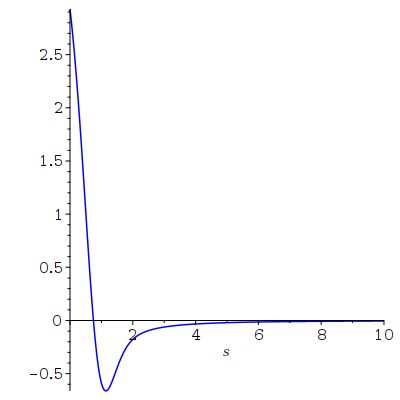}
		\label{f9b}
	}\quad
	\subfloat[]{\includegraphics[height=2.2 cm]{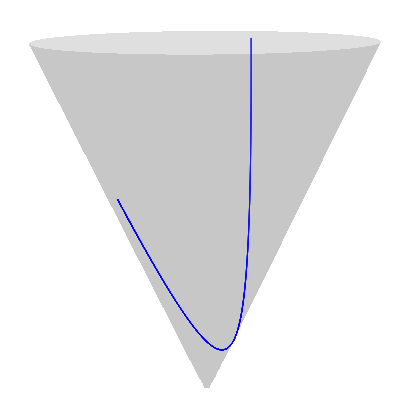}
		\label{f9c}
	}
	\subfloat[]{\includegraphics[height=2.1 cm]{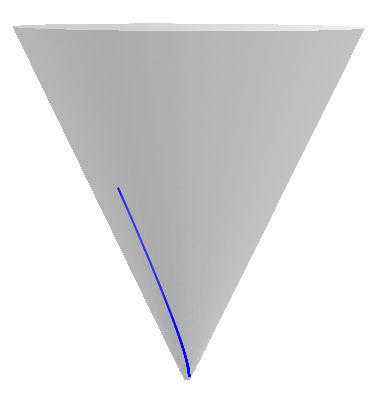}
		\label{f9d}}
	\caption{Self-similar solutions to the CF and ICF on $Q^{2}_{+}$ with $c=4$, $a=1$ and $v=(1,1,0)$.}
	\label{fig9}
\end{figure} 

Finally, in Figure \ref{f10a}, we visualize a self-similar solution $X$ to the CF on $Q_{+}^{2}$ when  $ c=3$, $ a=1$ and $v=(0,0,1)$ is a spacelike vector. Figure \ref{f10b} shows  the graph of $k(s)$, which vanishes at two points. Finally, in Figures \ref{f10c}, \ref{f10d} and \ref{f10e} we visualize the three components of the self-similar solutions to the ICF on $Q_{+}^{2}$ associated to $X$.

\begin{figure}[h!]
	\centering
	\subfloat[]{\includegraphics[height=2.0 cm]{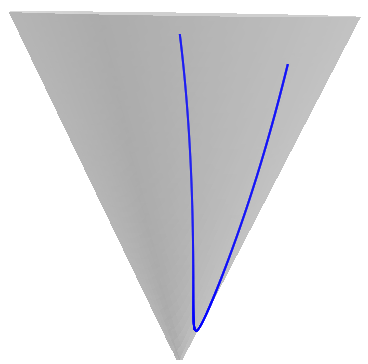}
		\label{f10a}
	}\quad
	\subfloat[]{\includegraphics[height=2.0 cm]{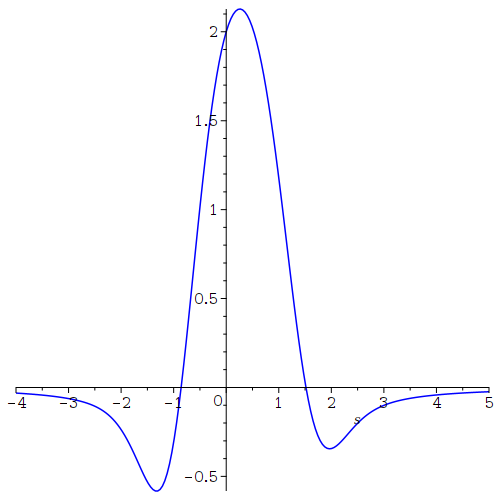}
		\label{f10b}
	}
	\subfloat[]{\includegraphics[height=1.95 cm]{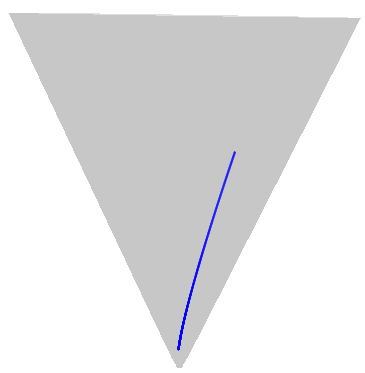}
		\label{f10c}
	}
	\subfloat[]{\includegraphics[height=1.9 cm]{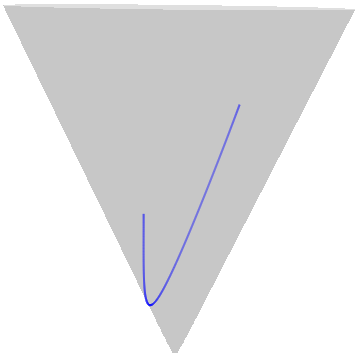}
		\label{f10d}}	
	\subfloat[]{\includegraphics[height=1.9 cm]{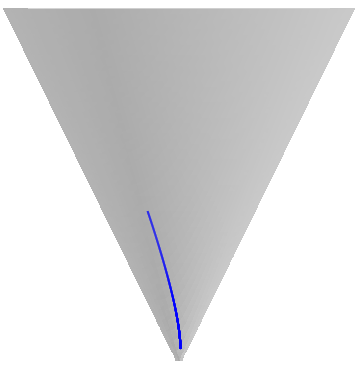}
		\label{f10e}}
	\caption{Self-similar solutions to the CF and ICF in $Q^{2}_{+}$ with $c=3$, $a=1$ and $v=(0,0,1)$.}
	\label{fig10}
\end{figure} 


\end{document}